\documentclass{amsart}

\usepackage{bm}
\usepackage{amsthm}
\usepackage{amsmath}
\usepackage{amssymb}
\usepackage{enumerate} 
\usepackage{comment}
\usepackage{mathtools}
\usepackage{url}
\usepackage{color}
\usepackage{enumitem}
\usepackage{booktabs}
\usepackage{placeins}

\theoremstyle{definition}
\newtheorem{dfn}{Definition}[section]
\newtheorem{prop}[dfn]{Proposition}

\newtheorem{lem}[dfn]{Lemma}

\newtheorem{thm}[dfn]{Theorem}

\theoremstyle{remark}
\newtheorem{rem}[dfn]{Remark}

\newcommand{\1}{
    \bm{1}
}

\newcommand{\pZ}[1]{
(\mathbb{Z}/{ #1 \mathbb{Z}})^{\times}
}

\newcommand{\aZ}[1]{
\mathbb{Z}/{ #1 \mathbb{Z}}
}

\title{On the level of distribution of Goldbach primes and its applications}

\author{Mizuki Akeno}
\date{\today}

\email{akeno.mizuki.tkb\_ee@u.tsukuba.ac.jp}
\date{\today}
\address{College of Mathematics, University of Tsukuba, Tsukuba, Japan}

\begin{document}

\begin{abstract}

We prove that, for almost all even integers $N>0$, the set of Goldbach primes $\mathbb{P} \cap (N-\mathbb{P})$ has a level of distribution $1/6$. 
As applications, we show that almost all even integers $N>0$ can be written as the sum of two primes $p_1, p_2$ such that $p_1-p_2+1 \in \mathbb{P}_4$.
We also prove an analogous result with $2p_1 p_2+1 \in \mathbb{P}_{13}$ for almost all integers \(N>0\) with \(6\mid N\). 
\end{abstract}

\maketitle

\section{Introduction}

Let $\mathbb{P}$ denote the set of primes, and let $\mathbb{P}_k$ denote the set of products of at most $k$ primes. 
It is well known that, for any $A>0$ and $X>2$, one has
\[ \sum_{m+n=N} \Lambda(m) \Lambda(n) = \mathfrak{S}(N) N + O\left( \frac{X}{(\ln{X})^A} \right) \]
for all but $O(X(\ln{X})^{-A})$ integers $N \in (X/2,X]$. Here, $\mathfrak{S}(N)$ is the singular series, defined by
\[ \mathfrak{S}(N) = \prod_{p} \left(1 + \frac{c_p(N)}{\varphi(p)^2} \right) = \sum_{q} \frac{\mu(q)^2}{\varphi(q)^2} c_q(N). \]

A Bombieri--Vinogradov type theorem for $\mathbb{P} \cap (N-\mathbb{P})$ has been studied in the following form:
\begin{equation} \sum_{d_1 \leq D_1} \sum_{d_2 \leq D_2} \lambda^{(1)}_{d_1} \lambda^{(2)}_{d_2} \left( \sum_{\substack{m+n=N \\ m \equiv l_1 \bmod d_1 \\ n \equiv l_2 \bmod d_2}} \Lambda(m) \Lambda(n) - M_0 \right) \ll \frac{X}{(\ln{X})^A} \label{BV1} \end{equation}
for almost all $N$, where $l_1, l_2$ are non-zero integers, $(\lambda^{(1)}_d),(\lambda^{(2)}_d)$ are divisor bounded sequences and $M_0=M_0(d_1,l_1,d_2,l_2,N)$ is a suitable approximation of the sum over $m,n$.

The case \(D_1=X^{1/2}\), \(D_2=1\) follows from the Bombieri--Vinogradov theorem and a standard circle method argument. See also \cite{kawada1993prime_ktuplets_ap},\cite{mikawa1992prime_twins_ap}. 
With a minor adaptation to almost-prime weights, this implies that, for almost all positive integers $N \equiv 4 \bmod 6$, the equation \(p_1+p_2=N\), \(p_1+2 \in \mathbb{P}_2\) is solvable.

Tolev's argument \cite{tolev2000additive} yields an estimate of the form \eqref{BV1} with $D_1=X^{1/2}, D_2 = X^{1/3}$. He applied this to show that, for almost all positive $N \equiv 4 \bmod 6$, the equation $p_1+p_2=N, p_1+2 \in \mathbb{P}_5, p_2+2 \in \mathbb{P}_7$ is solvable.

Matom{\"a}ki's argument \cite{matomaki2009bombieri} yields an estimate of the form \eqref{BV1} with $D_1=X^{1/2}, D_2 = X^{1/2}$, under the assumption that one of the sequences $\lambda^{(i)}_d$ is well-factorable. As an application, she proved the solvability of $p_1+p_2=N, p_1+2 \in \mathbb{P}_2, p_2+2 \in \mathbb{P}_7$ for almost all $N \equiv 4 \bmod 6$.

A related result for $p_1+p_2=N, p_1+2 \in \mathbb{P}_2, p_2+2 \in \mathbb{P}_2$ can be obtained, by different methods, from the work of Matom{\"a}ki and Shao \cite{matomaki2017vinogradov}. 
See also \cite{grimmelt2025exceptional} and \cite{grimmelt2022goldbach}.

One can extend \eqref{BV1} to the case where \(l_i\) is allowed to depend on \(d_i\) for each \(i=1,2\) and $D_1=X^{1/2}, D_2 = X^{1/3}$, with little additional effort.  See, for example, \cite[Theorem 2.1]{akeno2025small}. 
Thus, one may obtain similar results for the equation $p_1+p_2=N, f_1(p_1) \in \mathbb{P}_{r_1}, f_2(p_2) \in \mathbb{P}_{r_2}$ for suitable low-complexity functions $f_i$ such as polynomials and suitable $r_1,r_2 \geq 2$. 

The basic strategy in the proofs of the above-mentioned results is to rewrite the contribution of $\Lambda$ in \eqref{BV1} as
\[ \int_{\mathbb{T}} \left( \sum_{d_1 \leq D_1} \lambda_{d_1} \sum_{\substack{m \leq X \\ m \equiv l_1 \bmod d_1}} \Lambda(m) e(m\alpha) \right) \left( \sum_{d_2 \leq D_2} \lambda_{d_2} \sum_{\substack{n \leq X \\ n \equiv l_2 \bmod d_2}} \Lambda(n) e(n\alpha) \right) e(-N\alpha) d\alpha, \]
where $e(\cdot)=e^{2\pi i \cdot}$ and $\mathbb{T}=\mathbb{R}/\mathbb{Z}$, and then to use major- and minor-arc estimates for 
\[ \sum_{d_i \leq D_i} \lambda_{d_i} \sum_{\substack{n \leq X \\ n \equiv l_i \bmod d_i}} \Lambda(n) e(n\alpha). \]
Thus, in this approach the residue classes $l_i$ may depend on $d_i$, but not on $N$; this independence is used, in particular, when applying Bessel's inequality over the minor arc.

In this paper, we consider \eqref{BV1} in the case where $l_i$ may also depend on $N$. 
We prove the following result. 
\begin{thm}\label{T1}
Let $\theta \in (0,1/6)$, let $A>0$ and let $k \geq 0$. 
Let $w_1,w_2$ be smooth functions supported on $[\eta,1]$ for some $\eta>0$. 

For all integers $N \in (X/2,X]$ with $O(X(\ln{X})^{-A})$ exceptions, we have
\[ \sum_{d \leq X^{\theta}} \tau(d)^k \sup_{l:(d,l(N-l))=1} |R_{d,l}(N)| \ll \frac{X}{(\ln{X})^A} \]
where
\[ R_{d,l}(N) = \sum_{\substack{m+n=N \\ m \equiv l \bmod d}} \Lambda(m) \Lambda(n) w_1(m/X) w_2(n/X) - \frac{\mathfrak{S}(N)}{\psi_N(d)} \sum_{m+n=N} w_1(m/X) w_2(n/X) \]
and $\psi_N$ is defined by
\[ \psi_N(d) = d \prod_{\substack{p|d \\ p | N}} \left( 1 - \frac{1}{p} \right) \prod_{\substack{p|d \\ p \nmid N}} \left( 1 - \frac{2}{p} \right). \]
Here and throughout, we use the convention that the supremum over an empty set is $0$. 
\end{thm}

Indeed, we prove a more general version in which one of the $\Lambda$ factors is replaced by a more general sequence such as the indicator function of rough numbers. See Theorem \ref{T1p}.

By a slight modification of the result of Maier-Pomerance \cite[Theorem 3.1]{maier1990unusually} concerning the distribution of generalized twin primes in arithmetic progressions, one can establish such a result with $\theta < c$ for some absolute constant $c>0$. Their argument is based on the work of Montgomery and Vaughan \cite{montgomery1975exceptional} on the bounds for the exceptional set in the Goldbach conjecture. 
A computation by the author using standard zero-density estimates for Dirichlet $L$-functions shows that one may take $c = \frac{10}{314} = 0.0318\ldots$, although this value is probably not optimal.

We use methods that are different from those of \cite{maier1990unusually} in both major and minor arcs. 
We establish a new exponential sum estimate. See Theorem \ref{Lmb}. The exponent $1/6$ in Theorem \ref{T1} comes from this estimate. 
On the major arc, we use a mean value theorem of Choi and Kumchev \cite{choi2006mean} for Dirichlet polynomials, together with an argument of Maynard \cite{maynard2024note}.

As an application of Theorems \ref{T1} and \ref{T1p}, we consider the solvability of the equations of the form $p_1+p_2=N, f(p_1,p_2) \in \mathbb{P}_{r}$ for some function $f:\mathbb{Z} \times \mathbb{Z} \to \mathbb{Z}$ and $r \geq 2$.

Matom{\"a}ki and Z{\'u}{\~n}iga-Alterman \cite{matomaki2025weighted} showed that one can find $\mathbb{P}_3$ if the original and the switched problem have level of distribution at least $0.267$. 
We extend their result and show that one can find $\mathbb{P}_4$ if the original and the switched problem have level of distribution at least $0.1635$. See also the Appendix for similar results for $\mathbb{P}_k, 2 \leq k \leq 7$. 
Since $1/6=0.166\ldots>0.1635$, we obtain the following result. 
\begin{thm}\label{T2}
Let $A>0$ and $X \geq 2$. 
For all but $O(X(\ln{X})^{-A})$ even integers $N \in (X/2,X]$, there exist primes $p_1,p_2$ such that $p_1+p_2=N$ and $p_1-p_2+1$ has at most $4$ prime factors. 
\end{thm}
The numerical comparison in the Appendix indicates that the weights appearing in previous applications of the weighted sieve do not suffice for this purpose, since the resulting criteria remain above the level $1/6$. See the Appendix.

We also use Theorem \ref{T1} and Richert's weighted sieve to obtain
\begin{thm}\label{T3}
Let $A>0$ and $X \geq 2$. 
For all but $O(X(\ln{X})^{-A})$ integers $N \in (X/2,X]$ with $6|N$, there exist primes $p_1,p_2$ such that $p_1+p_2=N$ and $2p_1p_2+1$ has at most $13$ prime factors. 
\end{thm}
Note that the number $13$ arises as $13=2 \times 6 + 1$.

\subsection{AI tool disclosure}

OpenAI Codex (GPT-5.5) was used in the preparation of this manuscript in the following ways:
\begin{itemize}
\item assistance with the numerical computations reported in the Appendix and with the proof of the asymptotic statement there in the regime \(k\to\infty\);
\item assistance with searching for existing results relevant to the manuscript;
\item proofreading for minor typographical, stylistic, and mathematical errors.
\end{itemize}
Apart from these uses, the arguments and proofs in the paper were written by the author.  All AI-assisted material used in the proofs was checked by the author, who takes full responsibility for the content of the paper.

\section{Exponential sum estimate}

In this section, we prove the following estimate for exponential sums.

\begin{thm}\label{Lmb}
Let $D \leq X$ and $\varepsilon>0$. 
Let $w:\mathbb{R} \to \mathbb{R}$ be a smooth function supported on $[0,1]$. 
Let $\alpha \in [0,1], q \geq 1, (q,a)=1, |\alpha-a/q| \leq 1/q^2$. Then,
\begin{equation} \sum_{d \sim D} \sum_{b \in \aZ{d}} \left| \sum_{n} \Lambda(n) e((\alpha+b/d)n) w(n/X) \right|^2 \ll X^{2+\varepsilon} \left( \frac{D}{q} + \frac{qD}{X} + D^2 X^{-1/3} \right) \label{EXP} \end{equation}
\end{thm}

For the proof of Theorem \ref{T1}, a weaker form of \eqref{EXP} with $\aZ{d}$ replaced by $\pZ{d}$ would suffice. However, the argument below treats both cases in the same way, so we state the estimate in the present form.
In this setting, the large sieve gives the bound $\ll X^2(\ln X)^2$ when $D \leq X^{1/2}$, whereas Theorem \ref{Lmb} gives a sharper estimate when $q \in [X^{\varepsilon}D,X^{1-\varepsilon}D^{-1}]$ and $D \leq X^{1/6 - \varepsilon}$ for some $\varepsilon>0$.

The following two lemmas are fundamental. 

For $x \in \mathbb{R}$, let $\|x\|$ denote the distance from $x$ to the nearest integer. 
\begin{lem}\label{sumw}
Let $w:\mathbb{R} \to \mathbb{R}$ be a smooth, compactly supported function. 
Let $\alpha \in \mathbb{R}, \varepsilon>0$. Then
\begin{equation} \sum_{n} e(\alpha n) w(n/X) \ll X\#\{ n^* \in \mathbb{Z} : |\alpha - n^*| \leq X^{\varepsilon-1} \} + X^{-100} \label{ew_1} \end{equation}
and
\begin{equation} \sum_{n} e(\alpha n) w(n/X) \ll \frac{1}{\|\alpha\|}. \label{ew_2} \end{equation}
Here, the implicit constant depends only on $\varepsilon$ and $\{\|w^{(j)}\|_{\infty} : j \leq \lceil 200/\varepsilon \rceil \}$. 
\end{lem}

\begin{proof}
By the Poisson summation formula, we have
\[ \sum_{n} e(\alpha n) w(n/X) = X \sum_{n^*} \hat{w}(X(n^*-\alpha)), \]
where $\hat{w}(x) = \int_{\mathbb{R}} w(y) e(-xy) dy$. Integrating by parts $j$ times, we see that
\[ \hat{w}(x) = \frac{1}{(2\pi i x)^j} \int_{\mathbb{R}} w^{(j)}(y) e(-xy) dy \ll_j \|w^{(j)}\|_{\infty} |x|^{-j} \quad (x \neq 0) \]
and therefore
\[ X \hat{w}(X(n^*-\alpha)) \ll \frac{X}{(1+X|\alpha-n^*|)^j} \]
for any $j \geq 0$. Thus, the contribution from the terms with $|\alpha-n^*|>1/2$ is $O(X^{-100})$. 
If $|\alpha-n^*|\leq1/2$, we take $j=1, \lceil 200/\varepsilon \rceil$. This yields \eqref{ew_1} and \eqref{ew_2}. 
\end{proof}

\begin{lem}\label{bohr}
Let $\alpha \in \mathbb{R}, q \geq 1, (a,q)=1, |\alpha-a/q| \leq 1/q^2$ and $M,\delta>0$. 
\begin{equation} \#\{ M<m \leq 2M,n \in \mathbb{Z} : |\alpha - n/m| \leq \delta \} \ll M \left( \frac{1}{q} + q \delta + \delta M \right)  \label{qmn} \end{equation}
\end{lem}
\begin{proof}
If $M \delta \geq 1/4$, then there are $O(M\delta)$ integers $n$ for each $m$, and thus the claim follows. Thus, we may suppose that $M \delta < 1/4$. In this case, the left-hand side of \eqref{qmn} is bounded by
\[ \#\{ M<m \leq 2M : \| \alpha m\| \ll \delta M \} \]
since $n$ is unique for each $m$. Write $m=qs+r$ with $0< r \leq q$. Then $\|\alpha m\| =\|ar/q+(qs+r) \epsilon \|$, where $\epsilon=\alpha-a/q$.  

If $M \leq q/4$, then $s=0$ and $0<r\leq q/2$. Thus, we have
\[ \|\alpha m\| \geq \|ar/q\|-|\epsilon| q/2 \geq \|ar/q\|/2 \]
and therefore, there are $O(q \delta M)$ possible values of $r$. This gives \eqref{qmn}. 

If $M \geq q/4$, there are $O(M/q)$ possible values of $s$, and for each $s$, there are $q \delta M+1$ possible values of $r$. This gives \eqref{qmn}. 
\end{proof}

\begin{lem}\label{bohr_dn}
Let $\varepsilon>0$, and let $D,N \leq X$. Let $\alpha \in [0,1]$ with $q \geq 1, (a,q)=1$ and $|\alpha-a/q| \leq 1/q^2$. Then we have
\begin{equation} \label{exp_I_II} \#\{n \sim N, m \in \mathbb{Z},d \sim D, 1 \leq b \leq d : |\alpha+b/d + m/n| \leq \delta \} \ll X^{\varepsilon} ND \left( \frac{1}{q} + q \delta + \delta ND \right) \end{equation}
for all $\delta>0$. Here, the implicit constant depends only on $\varepsilon$. 
\end{lem}

\begin{proof}
Since $\alpha+b/d \ll 1$, we may suppose that $m \ll n$. 
For each $d,n$, let $g=(d,n)$ and write $d=d'g, n=n'g$. For each $c \ll d'n'g$, we consider the equation
\[ \frac{b}{d} + \frac{m}{n} = \frac{c}{d'n'g} \]
where $(b,m) \in \mathbb{Z}^2, b \ll d, m \ll n$. Since $bn' \equiv c \bmod d'$ and $(d',n')=1$, there are $\ll g$ possible values of $b$. For each $b$, $m$ is unique. Thus the number of solutions to this equation is $\ll g$. Hence, the left-hand side of \eqref{exp_I_II} is
\begin{align*}
&\ll (\ln{X}) \max_{G \ll \min(D,N)} G \#\{n' \sim N/G, d' \sim D/G, g \sim G, c \ll d'n'g : |\alpha+c/d'n'g| \leq \delta \}, \\
&\ll X^{\varepsilon} \max_{G \ll \min(D,N)} G \#\{m \sim ND/G, n \ll m : |\alpha+n/m| \leq \delta \}. 
\end{align*}
This and Lemma \ref{bohr} complete the proof. 
\end{proof}

We now prove Theorem \ref{Lmb}. 
The proof is based on a straightforward application of the above lemmas. 
A variant using large-sieve-type estimates might lead to an improved result.

\begin{proof}[Proof of Theorem \ref{Lmb}]

For any given $U,V \geq 1$, by Vaughan's identity, $\sum_{n} \Lambda(n) e((\alpha+b/d)n) w(n/X)$ can be written as a linear combination of $O((\ln{X})^2)$ sums of the form
\[ \sum_{\substack{m,n \\ m \sim M}} a_m e((\alpha+b/d)mn) w(mn/X), \]
where $M \leq UV$ and $(a_m)$ is a sequence with $|a_m| \leq X^{o(1)}$, and
\[ \sum_{\substack{m \sim M \\ n \sim N}} a_m b_n e((\alpha+b/d)mn) w(mn/X), \]
where $M \geq U, N \geq V, MN \asymp X$ and $(a_m),(b_n)$ are sequences with $|a_m|,|b_n| \leq X^{o(1)}$. Thus, setting $U=V=X^{1/3}$, the theorem follows from the Type I bound
\begin{equation}
\sum_{d \sim D} \sum_{b \in \aZ{d}} \left| \sum_{\substack{m,n \\ m \sim M}} a_m e((\alpha+b/d)mn) w(mn/X) \right|^2 \ll X^{2+2\varepsilon} \left( \frac{D}{q} + \frac{qD}{X} + \frac{MD^2}{X} \right), \label{exp_I}
\end{equation}
and the Type II bound
\begin{equation}
\sum_{d \sim D} \sum_{b \in \aZ{d}} \left| \sum_{\substack{m \sim M \\ n \sim N}} a_m b_n e((\alpha+b/d)mn) w(mn/X) \right|^2 \ll X^{2+2\varepsilon} \left( \frac{D}{q} + \frac{qD}{X} + \frac{MD^2}{X} + \frac{ND^2}{X} \right), \label{exp_II}
\end{equation}
for any given $\varepsilon>0$, $M,N$ with $MN \asymp X$ and sequences $(a_m),(b_n)$ with $|a_m|,|b_n|\leq 1$.

We show \eqref{exp_II}. By the Cauchy--Schwarz inequality and interchanging the order of summation, we have
\begin{align*}
\left| \sum_{\substack{m \sim M \\ n \sim N}} a_m b_n e((\alpha+b/d)mn) w(mn/X) \right|^2 &\leq M \sum_{m} \left| \sum_{n \sim N} b_n e((\alpha+b/d)mn) w(mn/X) \right|^2, \\
&\leq M \sum_{n_1,n_2 \sim N} \left| \sum_{m} e((\alpha+b/d)(n_1-n_2)m) w(mn_1/X) w(mn_2/X) \right|. \\
\end{align*}
Since $w(\cdot n_1/N) \times w(\cdot n_2/N)$ is a smooth, compactly supported function and its derivatives are bounded with constants independent of $n_1,n_2$, by Lemma \ref{sumw}, this is
\begin{align*}
&\ll M^2 \sum_{n_1,n_2 \sim N} \#\{m^* \in \mathbb{Z}:|(\alpha+b/d)(n_1-n_2)-m^*| \leq N/X^{1-\varepsilon} \} + X^{-10}, \\
&\ll XM + X M \#\left\{1 \leq n \leq N, m^* \in \mathbb{Z}:|\alpha+b/d - m^*/n| \leq \frac{N}{n X^{1-\varepsilon}} \right\} + X^{-10}.
\end{align*}
We divide $n$ into dyadic ranges, and use Lemma \ref{bohr_dn}. This gives
\begin{align*}
&\sum_{d \sim D} \sum_{b \in \aZ{d}} \#\left\{1 \leq n \leq N, m^* \in \mathbb{Z}:|\alpha+b/d - m^*/n| \leq \frac{N}{n X^{1-\varepsilon}} \right\}, \\
&\ll X^{2\varepsilon} \left( \frac{ND}{q} + \frac{qND}{X} + \frac{(ND)^2}{X} \right).
\end{align*}
This gives \eqref{exp_II}.

To prove \eqref{exp_I}, we simply apply
\[ \left| \sum_{\substack{m,n \\ m \sim M}} a_m e((\alpha+b/d)mn) w(mn/X) \right| \ll X, \]
and
\[ \left| \sum_{\substack{m,n \\ m \sim M}} a_m e((\alpha+b/d)mn) w(mn/X) \right| \ll \frac{X}{M} \#\{ m \sim M, n^* \in \mathbb{Z} : |\alpha+b/d-n^*/m| \ll X^{\varepsilon-1} \} + X^{-10}. \]
Thus the left-hand side of \eqref{exp_I} is bounded by
\[ \frac{X^2}{M} \#\{ d \sim D, b \in \aZ{d}, m \sim M, n^* \in \mathbb{Z} : |\alpha+b/d-n^*/m| \ll X^{\varepsilon-1} \} + X^{-1}. \]
This and Lemma \ref{bohr_dn} give \eqref{exp_I}. 
\end{proof}

\section{Level of distribution of Goldbach primes}

To prove Theorem \ref{T2}, we need Theorem \ref{T1} and an analogue of Theorem \ref{T1} for equations of the form $2m+n=N+1$, in which one of the $\Lambda$ factors is replaced by a structured weight. 
Since the proofs are essentially the same, we prove a generalized version in both settings. 

\begin{thm}\label{T1p}
Let $\theta \in (0,1/6)$ and let $\eta>0$ be sufficiently small. 
Let $k \geq 1$ and $\beta$ be a function of the form
\[ \beta(n) = \sum_{\substack{n_1 \cdots n_k=n \\ N_i<n_i \leq N'_i}} \Lambda(n_1) \cdots \Lambda(n_k), \]
where $N_1=1,N'_1=X$ if $k=1$, and $X^{\eta}<N_i <N'_i \leq 2N_i$, $N_1 \cdots N_k \leq X$ if $k>1$. 
Let $w_1,w_2$ be smooth functions supported on $[\eta,1]$. 

Let $A>0$ and $X \geq 2$. For all but $O(X(\ln{X})^{-A})$ integers $N \in (X/2,X]$, we have
\[ \sum_{d \leq X^{\theta}} \sup_{l:(d,l(N-l))=1} \left| \sum_{\substack{m+n=N \\ m \equiv l \bmod d}} \Lambda(m) \beta(n) w_1(m/X) w_2(n/X) - M_0 \right| \ll \frac{X}{(\ln{X})^A} \]
where $M_0=M_0(N,d)$ is given by
\[ M_0 = \frac{\mathfrak{S}(N)}{\psi_N(d)} \sum_{m+n=N} \beta(n) w_1(m/X) w_2(n/X), \quad \mathfrak{S}(N) = 2 \1_{2 | N} \prod_{p>2} \left( 1 + \frac{c_p(N)}{\varphi(p)^2} \right). \]

For all but $O(X(\ln{X})^{-A})$ integers $N \in (X/2,X]$, we have
\[ \sum_{\substack{d \leq X^{\theta} \\ 2 \nmid d}} \sup_{l:(d,l(N-2l))=1} \left| \sum_{\substack{2m+n=N \\ m \equiv l \bmod d}} \Lambda(m) \beta(n) w_1(m/X) w_2(n/X) - M'_0 \right| \ll \frac{X}{(\ln{X})^A} \]
where $M'_0=M'_0(N,d)$ is given by
\[ M'_0 = \frac{\mathfrak{S}'(N)}{2\psi_N(d)} \sum_{2m+n=N} \beta(n) w_1(m/X) w_2(n/X), \quad \mathfrak{S}'(N) = 2 \1_{2 \nmid N} \prod_{p>2} \left( 1 + \frac{c_p(N)}{\varphi(p)^2} \right). \]
\end{thm}

\begin{proof}[Proof of Theorem \ref{T1} assuming Theorem \ref{T1p}]
If $k=0$, the claim follows from Theorem \ref{T1p} with $\beta = \Lambda$ and the prime number theorem. 

By the triangle inequality, we have 
\[ R_{d,l}(N) \ll \frac{X}{d} (\ln{X})^2 \]
for all $d,l$ and $N \leq X$. Hence, the Cauchy--Schwarz inequality gives
\begin{align*}
\sum_{d \leq X^{\theta}} \tau(d)^k \sup_{l:(d,l(N-l))=1} |R_{d,l}(N)| &\ll \left( \sum_{d \leq X^{\theta}} \tau(d)^{2k} \frac{X}{d} (\ln{X})^2 \right)^{1/2} \left( \sum_{d \leq X^{\theta}} \sup_{l:(d,l(N-l))=1} |R_{d,l}(N)| \right)^{1/2} \\
&\ll (\ln{X})^{O_k(1)} \left( X \sum_{d \leq X^{\theta}} \sup_{l:(d,l(N-l))=1} |R_{d,l}(N)| \right)^{1/2}. \\
\end{align*}
This reduces the general case to the case $k=0$. 
\end{proof}

\subsection{Arithmetic lemmas}

For a Dirichlet character $\chi \bmod q$ and an integer $m$, define
\[ \tau_m(\chi) = \sum_{a \in \pZ{q}} \chi(a) e(-am/q). \]
Note that $\tau_m(\chi)=\overline{\chi}(-m) \tau(\chi)$ if $(m,q)=1$.  
\begin{lem}\label{tau}
Let $\chi$ be a Dirichlet character modulo $q$ with conductor $r$, and let $m$ be an integer. 
We have
\[ |\tau(\chi)| \leq r^{1/2} \]
and
\[ |\tau_m(\chi)| \leq (q(m,q))^{1/2} \]
\end{lem}

\begin{proof}
The first estimate is well-known. For the second estimate, see \cite[Lemma 5.4]{montgomery1975exceptional} or \cite[Lemma 4.1]{maynard2024note}. 
\end{proof}

For a Dirichlet character $\chi$, let $\chi^*$ denote the primitive character that induces $\chi$. 
\begin{lem}\label{sumA}
Let $d,q \geq 1$ and $N,l$ be integers such that $(d,l)=(d,N-l)=1$. Let $m,n$ be integers with $(m,[d,q])=1$. We have
\begin{align}
&\sum_{a \in \pZ{q}} \1_{m \equiv l \bmod d} e(a(m+n-N)/q) \notag \\
&= \sum_{\ell \ell^*|q} \frac{\mu(\ell^*)}{\varphi([d,q])\varphi(q/\ell)} \sum_{\chi_1 \bmod [d,q]} \sum_{\chi_2 \bmod q/\ell} \chi^*_2(\ell^*) \tau(\overline{\chi_2}) \chi^*_1(m) \1_{\ell \ell^*|n} \chi^*_2(n/\ell \ell^*) R(\chi_1,\chi_2,q,d,l,N), \label{eR}
\end{align}
where, for characters $\chi_1$ and $\chi_2$ with moduli dividing $[d,q]$ and $q$, respectively, we define $R=R(\chi_1,\chi_2,q,d,l,N)$ by
\[ R = \sum_{a \in \pZ{q}} \left( \sum_{\substack{b_1 \in \pZ{[d,q]} \\ b_1 \equiv l \bmod d}} \overline{\chi_1}(b_1) e(ab_1/q) \right) \chi_2(a) e(-aN/q). \] 
\end{lem}

\begin{proof}
Let $(n,q)=\ell$ and write $n=\ell n'$. Since $(m,[d,q])=1$ and $(n',q/\ell)=1$, the orthogonality of Dirichlet characters shows that the left-hand side of \eqref{eR} is
\begin{align*}
&\frac{1}{\varphi([d,q])\varphi(q/\ell)} \sum_{\chi_1 \bmod [d,q]} \sum_{\chi_2 \bmod q/\ell} \sum_{a \in \pZ{q}} \\
& \times \sum_{\substack{b_1 \in \pZ{[d,q]} \\ b_1 \equiv l \bmod d}} \sum_{b_2 \in \pZ{(q/\ell)}} \overline{\chi_1}(b_1) \overline{\chi_2}(b_2) e(a(b_1+\ell b_2-N)/q) \chi_1(m) \chi_2(n'). 
\end{align*}
Noting that
\[ \sum_{b_2 \in \pZ{(q/\ell)}} \overline{\chi_2}(b_2) e(ab_2/(q/\ell)) = \tau(\overline{\chi_2}) \chi_2(a) \]
and
\[ \chi_2(n') = \chi^*_2(n') \1_{(n/\ell,q/\ell)=1} = \sum_{\ell^*|q/\ell,n/\ell} \chi^*_2(\ell^*) \mu(\ell^*) \1_{\ell^*|n'} \chi^*_2(n'/\ell^*), \]
we obtain \eqref{eR}. 
\end{proof}

Note that $R(\chi_1,\chi_2,q,d,l,N)=R(\chi'_1,\chi'_2,q,d,l,N)$ if $(\chi_1)^*=(\chi'_1)^*$ and $(\chi_2)^*=(\chi'_2)^*$. 
The next lemma is essentially given in \cite{maier1990unusually}, with a slightly different definition of $R$. 
\begin{lem}\label{R}
Let $1 \leq d,q \leq N$ and $l$ be integers such that $(d,l)=(d,N-l)=1$. Let $\chi_1 \bmod [d,q], \chi_2 \bmod q$ be characters, and let $r_2$ denote the conductor of $\chi_2$. Then
\begin{equation} |R| \ll q \tau(q) (r_2,N)^{1/2} \ln{\ln{N}}, \label{R_bound}, \end{equation}
and, if $\chi_1$ and $\chi_2$ are principal characters, then
\begin{equation} R = \mu(q) c_{q/(d,q)}(N). \label{R_0} \end{equation}
\end{lem}

\begin{proof}
See \cite[Lemmas 11.1, 11.2, and 11.3]{maier1990unusually} for the proof of \eqref{R_bound}.

A proof of \eqref{R_0} for squarefree $q$ is also given in \cite[Section 8]{maier1990unusually}, but here we also need the non-squarefree case. 

Assume that $\chi_1$ and $\chi_2$ are principal characters. Then
\[ R = \sum_{a \in \pZ{q}} \left( \sum_{\substack{b \in \pZ{[d,q]} \\ b \equiv l \bmod d}} e(ab/q) \right) e(-aN/q). \]
Put $d^+=([d,q],d^{\infty}), q^{-}=[d,q]/d^+=q/(d^{\infty},q)$. Then $d|d^+, q^-|q, [d,q]=d^+ q^-$ and $(d^+,q^-)=1$. Thus
\[ \sum_{\substack{b \in \pZ{[d,q]} \\ b \equiv l \bmod d}} e(ab/q) = \sum_{c=1}^{d^+/d} \sum_{\substack{b \in \pZ{[d,q]} \\ b \equiv l+dc \bmod d^+}} e(ab/q). \]
We write $b= (l+dc) q^- \overline{q^-} + d^+ b'$ where $q^- \overline{q^-} \equiv 1 \bmod d^+$. This gives
\[ \sum_{\substack{b \in \pZ{[d,q]} \\ b \equiv l+dc \bmod d^+}} e(ab/q) = e(a (l+dc) q^- \overline{q^-}/q) \sum_{b' \in \pZ{q^-}} e(ad^{+} b'/q). \]
Since $d^+/q=\frac{d/(d,q)}{q^-}$ and $(d/(d,q),q^-)=1$, we see that
\[
\sum_{b' \in \pZ{q^-}} e(ad^{+} b'/q) = \mu(q^-)
\]
and
\[ R = \mu(q^-)  \sum_{c=1}^{d^+/d} \sum_{a \in \pZ{q}} e(a ((l+dc) q^- \overline{q^-}-N)/q) = \mu(q^-)  \sum_{c=1}^{d^+/d} c_q((l+dc) q^- \overline{q^-}-N). \]
Noting that $((l+dc) q^- \overline{q^-}-N,p)=(l-N,p)=1$ for all $p|(q,d^{\infty})$, we have
\begin{align*}
c_q((l+dc) q^- \overline{q^-}-N) &= c_{q^-}((l+dc) q^- \overline{q^-}-N) c_{(q,d^{\infty})}((l+dc) q^- \overline{q^-}-N), \\
&= c_{q^-}(N) \mu((q,d^{\infty})) = c_{q/(d,q)}(N) \mu((q,d^{\infty})). 
\end{align*}
This gives \eqref{R_0}. 
\end{proof}

\begin{rem}
Let $r_1$ be the conductor of $\chi_1$, and let $r^-_1=(r_1,d^{\infty})$. Following the above argument and bounding the resulting sum over $c$ trivially, we obtain
\[ |R| \leq \frac{(q,d^{\infty})}{(q,d)} (N,q^-) (r^-_1 r_2 [r^-_1,r_2])^{1/2}. \]
However, this estimate can be worse than \eqref{R_bound} when $(q,d^\infty)$ is large.
\end{rem}

\begin{lem}\label{sing}
Let $d,N \geq 1$ with $\psi_N(d) \neq 0$. We have
\[ \sum_{q} \frac{\mu(q)^2 \varphi((d,q))}{\varphi(d) \varphi(q)^2} c_{q/(d,q)}(N) = \frac{\mathfrak{S}(N)}{\psi_N(d)} . \]
\end{lem}

\begin{proof}
Since the summand is multiplicative with respect to $q$, the left-hand side equals 
\[ = \frac{1}{\varphi(d)} \prod_{p|d} \left( 1 + \frac{1}{\varphi(p)} \right) \prod_{p \nmid d} \left( 1 + \frac{c_{p}(N)}{\varphi(p)^2} \right). \]
Moreover,
\[ \frac{1}{\varphi(d)} \prod_{p|d} \left( 1 + \frac{1}{\varphi(p)} \right)\left( 1 + \frac{c_{p}(N)}{\varphi(p)^2} \right)^{-1} = \frac{1}{d} \prod_{\substack{p|d \\ p|N}} \frac{p}{p-1} \prod_{\substack{p|d \\ p\nmid N}} \frac{p}{p-2} = \frac{1}{\psi_N(d)}. \]
This completes the proof.  
\end{proof}

\subsection{Analytic Lemmas}

The following classical result follows from Gallagher's lemma. See also \cite[Proposition 5.1]{matomaki2019correlations}. 
\begin{lem}\label{Gal}
Let $\varepsilon>0, x \geq 1$, and $\eta>0, \delta \geq 1/(\eta x)$. 
Let $f$ be an arithmetic function supported on $[1,x/\eta]$ and satisfying $f(n) \ll x$, and let $w$ be a smooth function supported on $[\eta,1]$. Then
\begin{equation}
\int_{-\delta}^{\delta} \left| \sum_{n} f(n) e(n\epsilon) w(n/x) \right|^2 d\epsilon \ll \delta^2 x \sup_{\substack{I \subset [\eta x,x] \\ |I| \leq 2\delta^{-1}}} \left| \sum_{n \in I} f(n) \right|^2 \label{Gal1}
\end{equation}
and
\begin{equation}
\int_{-\delta}^{\delta} \left| \sum_{n} f(n) e(n\epsilon) w(n/x) \right|^2 d\epsilon \ll x^{-1} \left( \int_{|t| \leq \delta x^{1+\varepsilon}} \left| \sum_{n} \frac{f(n)}{n^{it}} \right| dt \right)^2 + x^{-10}. \label{Gal2}
\end{equation}
Here, the implicit constant depends only on $\varepsilon,\eta$ and $w$. 
\end{lem}

\begin{proof}

Let $\phi:\mathbb{R} \to \mathbb{R}$ be a smooth even function supported on $[-1,1]$ and satisfying $\hat{\phi}(u) = \int_{-1}^{1} \phi(y) \cos(2\pi uy) dy \geq 1$ for $u \in [-1,1]$. Then 
\begin{equation}
\int_{-\delta}^{\delta} \left|\sum_{n} f(n) e(n\epsilon) w(n/x)\right|^2 d\epsilon \leq \int_{-\infty}^{\infty} \left|\sum_{n} f(n) e(n\epsilon) w(n/x)\right|^2 |\hat{\phi}(\epsilon/\delta)|^2 d\epsilon. \label{Gal_1}
\end{equation}
Since
\begin{align*}
\sum_{n} f(n) e(-n\epsilon) w(n/x) \hat{\phi}(\epsilon/\delta) &= \sum_{n} f(n) e(-n\epsilon) w(n/x) \int_{-\infty}^{\infty} \phi(y) e(-\epsilon y/\delta) dy \\
&= \delta \int_{-\infty}^{\infty} \left( \sum_{n} f(n) w(n/x) \phi(\delta(z-n)) \right) e(-\epsilon z) dz,
\end{align*}
it follows from the Plancherel identity that the right-hand side of \eqref{Gal_1} is
\[ \delta^2 \int_{-\infty}^{\infty} \left| \sum_{n} f(n) w(n/x) \phi(\delta(y-n)) \right|^2 dy. \]
We bound this trivially by
\[ \ll \delta^2 x \sup_{y \in \mathbb{R}} \left| \sum_{n} f(n) w(n/x) \phi(\delta(y-n)) \right|^2. \]
This, together with partial summation, gives \eqref{Gal1}.

Fix $y \in \mathbb{R}$ and let
\[ W(s) = \int_{0}^{\infty} w(z/x) \phi(\delta(y-z)) z^{s-1} dz. \]
Trivially, 
\begin{equation} W(it) \ll \int_{0}^{\infty} |\phi(\delta(y-z))| x^{-1} dz \ll (\delta x)^{-1} \label{W_t} \end{equation}
for all $t \in \mathbb{R}$. 
Integrating by parts $j\geq 0$ times, we obtain 
\begin{align*}
W(it) &= (-1)^j \int_{\mathbb{R}} \frac{z^{it+j}}{(it+1)(it+2)\cdots (it+j)} \left( \frac{d^j}{dz^j} w(z/x) \phi(\delta(y-z)) z^{-1} \right) dz, \\
&\ll_j \int_{\eta x}^{x} \frac{x^j}{(1+|t|)^j} (1/x+\delta)^j x^{-1} dz \ll \left( \frac{x\delta}{1+|t|} \right)^j,
\end{align*}
for any $j\geq 0$ and $t \in \mathbb{R}$. 
This gives
\begin{equation} W(it) \ll \frac{x^{-100}}{|t|^2} \label{W_100} \end{equation}
for $|t| \geq \delta x^{1+\varepsilon}$. 

By the Mellin inversion formula, we have
\[ w(n/x) \phi(\delta(y-n)) = \frac{1}{2\pi} \int_{-\infty}^{\infty} W(it) n^{-it} dt \]
and thus
\[ \sum_{n} f(n) w(n/x) \phi(\delta(y-n)) = \frac{1}{2\pi} \int_{-\infty}^{\infty} W(it) \left( \sum_{n} \frac{f(n)}{n^{it}} \right) dt. \]
Using \eqref{W_100} and \eqref{W_t}, we obtain \eqref{Gal2}.

\end{proof}

\begin{lem}\label{D0}
Let $\beta$ be as in Theorem \ref{T1p}, and let $1 < h \leq X^{1/2}$. Then
\begin{equation} \label{D0eqn} \beta(n) \leq (\ln{n})^k \end{equation}
and
\begin{equation} \label{D0eq} \sum_{n} \beta(hn) \ll X^{1-\eta} (\ln{X})^{O(1)} \end{equation}
\end{lem}

\begin{proof}

If $k=1$, both \eqref{D0eqn} and \eqref{D0eq} are trivial.

Assume that $k \geq 2$. Recall that
\[ \beta(n) = \sum_{\substack{n_1 \cdots n_k = n \\ N_i < n_i \leq N'_i}} \Lambda(n_1) \cdots \Lambda(n_k). \]
Since $n_i|n$, \eqref{D0eqn} follows from
\[ \beta(n) \leq \left( \sum_{d|n} \Lambda(d) \right)^k = (\ln{n})^k. \]

It remains to prove \eqref{D0eq}. 

Let $f:\mathbb{N}^k \to \mathbb{C}$. We first note the identity
\begin{equation} \label{hrem} \sum_{n_1 \cdots n_k = hn} f(n_1,\ldots,n_k) = \sum_{\substack{g_i r_i|h \\ h|g_1 \cdots g_k}} \mu(r_1) \cdots \mu(r_k) \sum_{\frac{g_1 \cdots g_k}{h} r_1 \cdots r_k n_1 \cdots n_k = n} f(g_1r_1n_1,\ldots,g_kr_kn_k). \end{equation} 
We write $(n_i,h)=g_i$. Then we see that $h|g_1 \cdots g_k$ and $(n_i/g_i,h/g_i)=1$. We use the Möbius inversion formula 
\[ \1_{(n_i/g_i,h/g_i)=1} = \sum_{r_i|n_i/g_i,h/g_i} \mu(r_i), \]
and write $n_i=r_i g_i n'_i$. This gives \eqref{hrem}. 

Let $h \leq X^{1/2}$. We may assume that $h$ is a product of at most $k$ prime powers: otherwise $\beta(h \cdot)=0$. By \eqref{hrem}, the left-hand side of \eqref{D0eq} can be written as a linear combination of $(\ln{X})^{O_k(1)}$ sums of the form
\[ \prod_{i=1}^{k} \sum_{N_i/q_i< n_i \leq N'_i/q_i} \Lambda(q_i n_i), \]
where $h|q_1 \cdots q_k$. 
The claim then follows from the bounds
\[ \sum_{N_i/q_i< n_i \leq N'_i/q_i} \Lambda(q_i n_i) \ll
\begin{cases}
N_i & q_i=1, \\
\ln{X} & q_i>1,
\end{cases}
\]
together with $N_i \geq X^{\eta}$. 
\end{proof}

\begin{lem}\label{D1}
Let $\beta$ be as in Theorem \ref{T1p}. Then
\begin{equation} \label{D1eq} \sum_{\substack{r \leq R \\ d|r}} \sideset{}{^*}\sum_{\chi \bmod r} \int_{-T}^{T} \left| \sum_{n} \beta(hn) \chi(n) n^{it} \right| dt \ll \left(\frac{X}{h} + \frac{R^2T}{d}\left( \frac{X}{h} \right)^{11/20}\right) (\ln{X})^{O(1)} \end{equation}
for all $d,R \leq X$, $h \leq X^{1/2}$, and $T \geq 1$. 
\end{lem}

\begin{proof}
If $k=1$, this follows from \cite[Theorem 1.1]{choi2006mean}. 

Assume that $k \geq 2$. We claim that there exist $\ell \ll k$ and $(M_j)_{1 \leq j \leq 2\ell}, (t_j)_{1 \leq j < 2\ell}$ such that
\[ M_1 \cdots M_{2\ell} \ll X/h, \quad M_1 \leq \cdots \leq M_{\ell} \leq X^{1/100}, \quad |t_1|,\ldots,|t_{2\ell-1}| \leq X^{100} \]
and the left-hand side of \eqref{D1eq} is bounded by
\begin{equation} \label{D1F} (\ln{X})^{O_k(1)} \sum_{\substack{r \leq R \\ d|r}} \sideset{}{^*}\sum_{\chi \bmod r} \int_{-T}^{T} |F_1(t) \cdots F_{2\ell}(t)| dt + X^{-10}, \end{equation}
where
\[ F_j(t) = \sum_{n_j \sim M_j} \mu(n_j) \chi(n_j) n_j^{i(t+t_j)} \]
for $1 \leq j \leq \ell$, 
\[ F_j(t) = \sum_{n_j \sim M_j} f(n_j) \chi(n_j) n_j^{i(t+t_j)}, \quad f \in \{1,\ln{\cdot}\} \]
for $\ell+1 \leq j \leq 2\ell-1$ and
\[ F_j(t) = \int_{|t_{2\ell}| \leq X^{100}} \left| \sum_{n_j \sim M_j} f(n_{j}) \chi(n_j) n_j^{i(t+t_{2\ell})} \right| \frac{dt_{2\ell}}{1+|t_{2\ell}|}, \quad f \in \{1,\ln{\cdot}\}, \]
for $j=2\ell$.

We use \eqref{hrem} and estimate the terms with \(q_i>1\) trivially. Thus, we see that $\left| \sum_{n} \beta(hn) \chi(n) n^{it} \right|$ is bounded by a linear combination of $(\ln{X})^{O_k(1)}$ sums of the form
\[ \prod_{i=1}^{I} \left| \sum_{\substack{n_i \\ L_i < n_i \leq L'_i}} \Lambda(n_i) \chi(n_i) n^{it}_i \right| \]
where $I \leq k, \{(L_i,L'_i)\} \subset \{(N_i,N'_i)\}, L_1 \cdots L_I \ll X/h$.

For each $i$, we apply Heath-Brown's identity to $\Lambda(n_i)$ and Perron's formula. 
Noting that $\left| ({L'_i}^{s}-{L_i}^{s})/s \right| \ll 1/(1+|s|)$, we see that $\left| \sum_{L_i < n_i \leq L'_i} \Lambda(n_i) \chi(n_i) n^{it}_i \right|$ is bounded by $(\ln{X})^{O_k(1)}$ integrals of the form 
\[ \int_{-X^{100}}^{X^{100}} \prod_{j=1}^{200} \left| \sum_{n_{i,j} \sim L_{i,j}} f_j(n_{i,j}) \chi(n_{i,j}) n_{i,j}^{i(t'+t)} \right| \frac{1}{1+|t'|} dt' + X^{-50} \]
where $L_{i,1} \cdots L_{i,200} \ll L_i$, $L_{i,j} \leq X^{1/100}$ for $1 \leq j \leq 100$, and 
\[ f_j(n) =
\begin{cases}
\mu(n) & 1 \leq j \leq 100, \\
\ln{n} & j=101, \\
1 & 101 < j \leq 200.
\end{cases}
\]
Inserting these decompositions and interchanging the order of integration, we obtain the claim.

It remains to estimate \eqref{D1F}. This follows from \cite[Theorem 2.1]{choi2006mean} and the argument in the proof of \cite[Theorem 1.1]{choi2006mean}. The only minor difference is the additional integral over \(t_{2\ell}\) in the last factor. 
After inserting Perron's formula for the \(n_{2\ell}\)-sum, this \(t_{2\ell}\)-integral is bounded trivially, and the remaining argument is the same.
\end{proof}

\begin{lem}\label{D21}
Let $A,\varepsilon>0$. We have
\begin{equation} \sum_{x-y<n \leq x} (\Lambda(n) - 1) \ll \frac{y}{(\ln{x})^A} \end{equation}
for $x \geq 2, y \in [x^{7/12+\varepsilon},x]$. 
\end{lem}

\begin{proof}
See, for example, \cite{heath1982prime}. 
\end{proof}

\begin{lem}\label{D22}
Let $\beta$ be an arithmetic function as in Theorem \ref{T1p}. 

Let $A,\varepsilon>0$ and $\chi \neq 1$ be a non-principal primitive Dirichlet character modulo $\ll (\ln{X})^A$. Let $h \leq X^{1/2}$ and $T \ll (X/h)^{5/12-\varepsilon}$. 
We have
\begin{equation} \int_{-T}^{T} \left| \sum_{n} \beta(hn) \chi(n) n^{it} \right| dt \ll \frac{X}{h(\ln{X})^A}. \label{short_chi} \end{equation} 
\end{lem}

\begin{proof}
If $k=1$, then we may suppose that $h=1$. The claim follows from an argument similar to that in \cite{heath1982prime} with the following modifications:
\begin{itemize}
\item Since \(L(s,\chi)\) has no pole at \(s=1\), an argument similar to that in \cite[Lemma 2.3]{choi2006mean} allows us to treat also the range \(M_j \geq X^{1/2}\) by the method of Section 4 of \cite{heath1982prime}. 
\item We can include the integration over $|t| \leq \exp((\ln{X})^{1/3})$, since the Siegel--Walfisz theorem gives
\[ \sum_{n \sim M} f(n) \chi(n) n^{it} \ll M \exp\left( - c (\ln{X})^{1/2} \right), \quad f \in \{1,\ln,\mu\} \]
in this range if $X^{\varepsilon} \leq M \leq X$. 
\end{itemize}

If $k\geq 2$, the claim follows from the argument used in Lemma \ref{D1} and \cite{heath1982prime} with the above modifications. 
\end{proof}

\subsection{Proof of Theorem \ref{T1p}}

Let $\theta \in (0,1)$ and set $D=X^{\theta}$. 

Let $\beta$ be an arithmetic function satisfying the conditions in Theorem \ref{T1p}. Let $\Lambda'(n)=\1_{\mathbb{P}}(n) \ln{n}$.

For $d \geq 1, l \bmod d$ and Dirichlet character $\chi$, we write
\[ S^{(1)}_{l,d}(\alpha,\chi) = \sum_{n \equiv l \bmod d} \Lambda'(n) \chi(n) e(\alpha n) w_1(n/X), \quad S^{(2)}_{d}(\alpha,\chi) = \sum_{n} \beta(nd) \chi(n) e(\alpha nd) w_2(nd/X). \]
We omit $\chi$ if $\chi=1$, and $d,l$ if $d=l=1$.

By orthogonality, we have
\begin{equation} \sum_{\substack{m+n=N \\ m \equiv l \bmod d}} \Lambda(m) \beta(n) w_1(m/X) w_2(n/X) = \int_{\mathbb{T}} S^{(1)}_{l,d}(\alpha) S^{(2)}(\alpha) e(-N\alpha) d\alpha + O(X^{1/2+\varepsilon} ). \label{Morth} \end{equation}
Let $\varepsilon>0$ and set $P=DX^{\varepsilon}$. We define the major arc
\[ \mathfrak{M} = \bigcup_{q \leq P} \bigcup_{a \in \pZ{q}} \mathfrak{M}_{q,a}, \quad \mathfrak{M}_{q,a} = \{ \alpha \in \mathbb{T} : \| \alpha - \frac{a}{q} \| \leq \frac{P}{qX} \} \]
and the minor arc $\mathfrak{m} = \mathbb{T} \setminus \mathfrak{M}$. 
We assume $\theta < 1/2-\varepsilon$ so that the arcs $\mathfrak{M}_{q,a}$ are disjoint. 
Note that for $\alpha \in \mathfrak{m}$, there exist $P \leq q \leq XP^{-1}, (a,q)=1$ such that $|\alpha-a/q| \leq 1/q^2$ by Dirichlet's approximation theorem and the definition of $\mathfrak{m}$.

\begin{lem}[Minor arc bound]\label{Min}
If $\theta \leq 1/6-\varepsilon$, then
\begin{equation}
\sum_{X/2 < N \leq X} \left( \sum_{d \leq D} \sup_{l} \left| \int_{\mathfrak{m}} S^{(1)}_{l,d}(\alpha) S^{(2)}(\alpha) e(-N\alpha) d\alpha \right| \right)^2 \ll X^{3-\varepsilon/2}. \label{mb}
\end{equation}
\end{lem}

\begin{proof}
By the orthogonal relation
\[ \1_{n \equiv l \bmod d} = \frac{1}{d} \sum_{b \in \aZ{d}} e(b(n-l)/d) \]
and the triangle inequality, we have
\[ \sup_{l} \left| \int_{\mathfrak{m}} S^{(1)}_{l,d}(\alpha) S^{(2)}(\alpha) e(-N\alpha) d\alpha \right| \leq \frac{1}{d} \sum_{b \in \aZ{d}}  \left| \int_{\mathfrak{m}} S^{(1)}(\alpha+b/d) S^{(2)}(\alpha) e(-N\alpha) d\alpha \right|. \]
Writing $b/d=b'/d'$ with $(b',d')=1$, we have
\begin{align*}
& \sum_{d \leq D} \frac{1}{d} \sum_{b \in \aZ{d}}  \left| \int_{\mathfrak{m}} S^{(1)}(\alpha+b/d) S^{(2)}(\alpha) e(-N\alpha) d\alpha \right|, \\
&\ll (\ln{X}) \sum_{d' \leq D} \frac{1}{d'} \sum_{b' \in \pZ{d'}} \left| \int_{\mathfrak{m}} S^{(1)}(\alpha+b'/d') S^{(2)}(\alpha) e(-N\alpha) d\alpha \right|. 
\end{align*}
Thus, using the Cauchy--Schwarz inequality and Bessel's inequality, we see that the left-hand side of \eqref{mb} is bounded by 
\begin{align}
&\ll (\ln{X})^2 \sum_{X/2 < N \leq X} \left( \sum_{d \leq D} \sum_{b \in \pZ{d}} \left| \int_{\mathfrak{m}} S^{(1)}(\alpha+b/d)S^{(2)}(\alpha) e(-N\alpha) d\alpha \right|^2 \right) \left( \sum_{d \leq D} \sum_{b \in \aZ{d}} \frac{1}{d^2} \right), \notag \\
&\ll (\ln{X})^3 \sum_{d \leq D} \sum_{b \in \pZ{d}} \int_{\mathfrak{m}} |S^{(1)}(\alpha+b/d)|^2 |S^{(2)}(\alpha)|^2 d\alpha. \label{dbS}
\end{align}
After discarding prime powers trivially, we use Theorem \ref{Lmb} and 
\[ \int_{\mathbb{T}}  |S^{(2)}(\alpha)|^2 d\alpha \ll X(\ln{X})^{O(1)}. \]
This gives \eqref{mb}. 

\end{proof}

\begin{rem}
In \cite{maier1990unusually}, Maier and Pomerance used classical minor-arc bounds for $S^{(2)}(\alpha)$ in the case $\beta=\Lambda$. Their method would\footnote{In the step corresponding to \cite[(6.6)]{maier1990unusually}, the supremum over the residue class should be retained before applying Bessel's inequality.} correspond to bounding \eqref{dbS} by
\begin{align*}
& \left( \sup_{\alpha \in \mathfrak{m}} |S^{(2)}(\alpha)|^2 \right) \sum_{d \leq D} \sum_{b \in \pZ{d}} \int_{\mathbb{T}} |S^{(1)}(\alpha+b/d)|^2 d\alpha, \\
&\ll \left( \frac{X^2}{P} + X^{8/5} + XP \right) X D^2 (\ln{X})^{O(1)}. 
\end{align*}
This is sufficient on the minor arcs for $D \leq X^{1/5-\varepsilon}$, but only with $P=X^{\varepsilon}D^2$.
\end{rem}

\begin{lem}[Reduction for the major arc]\label{Mred}
We have
\begin{align*}
&\sum_{X/2 < N \leq X} \sum_{d \leq D} \sup_{l:(d,l(N-l))=1} \left| \int_{\mathfrak{M}} S^{(1)}_{l,d}(\alpha) S^{(2)}(\alpha) e(-N\alpha) d\alpha - M_0 \right| \\
&\ll X^{2} (\ln{X})^{O(1)} (X^{-\eta/2} + P^{-1/2}) + X (\ln{X})^{O(1)} \sup_{Q \ll P} (X^{1/2} E_1+E_2),
\end{align*}
where $E_i=E_i(Q)$ are defined by
\[ E_1 = Q^{-1} \left( \int_{|\epsilon| \leq P/(QX)} |S^{(1)}(\epsilon;\Lambda'-1)|^2 d\epsilon \right)^{1/2} \]
\begin{align*}
E_2 &= \underset{r_1r_2 \neq 1}{\sum_{r_1 \leq QD} \sum_{h r_2 \leq Q}} \sideset{}{^*}\sum_{\xi_1 \bmod r_1} \sideset{}{^*}\sum_{\xi_2 \bmod r_2} \frac{h r^{1/2}_2}{[r_1,h r_2]} (\tau(r_1) \tau(r_2) \tau(h))^{4} \\
& \times \int_{|\epsilon| \leq P/(QX)} |S^{(1)}(\epsilon;\xi_1) S^{(2)}_{h}(\epsilon;\xi_2)| d\epsilon
\end{align*}
and we write 
\[ S^{(i)}_h(\epsilon;f) = \sum_{n} f(nh) e(nh \epsilon) w_i(nh/X), \quad S^{(i)}(\epsilon;f) = S^{(i)}_1(\epsilon;f), \]
for an arithmetic function $f$ and $i=1,2$.  
\end{lem}

\begin{proof}

Let $M=M(d,l,N)$ denote
\begin{align*}
M &= \int_{\mathfrak{M}} S^{(1)}_{l,d}(\alpha) S^{(2)}(\alpha) e(-N\alpha) d\alpha \\
&= \sum_{q \leq P} \sum_{a \in \pZ{q}} e(-Na/q) \int_{|\epsilon| \leq P/(qX)} S^{(1)}_{l,d}(\epsilon+a/q) S^{(2)}(\epsilon+a/q) e(-N\epsilon) d\epsilon. 
\end{align*}

By Lemma \ref{sumA}, we have
\begin{align*}
&\sum_{a \in \pZ{q}} e(-Na/q) S^{(1)}_{l,d}(\epsilon+a/q) S^{(2)}(\epsilon+a/q) \\
&= \sum_{\ell \ell^*|q} \frac{\mu(\ell^*)}{\varphi([d,q]) \varphi(q/\ell)} \sum_{\chi_1 \bmod [d,q]} \sum_{\chi_2 \bmod q/\ell} \chi^*_2(\ell^*) \tau(\overline{\chi_2}) R(\chi_1,\chi_2;q,d,l,N) S^{(1)}(\epsilon;\chi^*_1) S^{(2)}_{\ell \ell^*}(\epsilon;\chi^*_2).
\end{align*}
We write $(\chi_1)^*=\xi_1 \bmod r_1$ and $(\chi_2)^*=\xi_2 \bmod r_2$. Note that $(\ell^*,r_2)=1$ and $r_2|q/\ell$ imply $\ell \ell^* r_2|q$. Thus
\begin{align*}
M &= \sum_{q \leq P} \sum_{r_1|[d,q]} \sum_{\ell \ell^* r_2|q} \frac{\mu(\ell^*)}{\varphi([d,q]) \varphi(q/\ell)} \sideset{}{^*}\sum_{\xi_1 \bmod r_1} \sideset{}{^*}\sum_{\xi_2 \bmod r_2} \xi_2(\ell^*) \tau(\overline{\xi_2} \chi_0) R(\xi_1,\xi_2,q,d,l,N) \\
& \times \int_{|\epsilon| \leq P/(qX)} S^{(1)}(\epsilon;\xi_1) S^{(2)}_{\ell \ell^*}(\epsilon;\xi_2) e(-N\epsilon) d\epsilon,
\end{align*}
where $\chi_0$ is the principal character modulo $q/\ell$.

Let $M_1, M_2$ be the contributions from the term with $r_1=r_2=1$ and $r_1 r_2 \neq 1$ to $M$ respectively.

\vskip\baselineskip

We show that 
\begin{equation} \sum_{N \leq X} \sum_{d \leq D} \sup_{l:(d,l(N-l))=1} |M_1 - M_0| \ll X^{2} (\ln{X})^{O(1)} (X^{-\eta/2} + P^{-1/2}) + X^{3/2} (\ln{X})^{O(1)} \sup_{Q \leq P} E_1 \label{11-0}. \end{equation}

By Lemma \ref{R}, we have 
\begin{align*}
M_{1} &= \sum_{q \leq P} \sum_{\ell \ell^*|q} \frac{\mu(q)^2 \mu(\ell \ell^*) \varphi((d,q)) \varphi(\ell)}{\varphi(d) \varphi(q)^2} c_{q/(d,q)}(N) \int_{|\epsilon| \leq P/(qX)} S^{(1)}(\epsilon) S^{(2)}_{\ell \ell^*}(\epsilon) e(-N\epsilon) d\epsilon. 
\end{align*}
Let $M_{11}, M_{12}$ denote the contributions from the terms with $\ell \ell^*=1$ and $\ell \ell^* \neq 1$ to $M_{1}$, respectively.

Using $|c_{q/(d,q)}(N)| \leq (q,N)$ and standard upper and lower bounds for $\varphi$, we have
\[ M_{12} \ll \ln{X} \sum_{q \leq P} \sum_{\ell \ell^*|q} \frac{\ell (q,N) (d,q)}{dq^2} \int_{|\epsilon| \leq 1/2} |S^{(1)}(\epsilon) S^{(2)}_{\ell \ell^*}(\epsilon)| d\epsilon. \]
Using the Cauchy--Schwarz inequality and Lemma \ref{D0}, we obtain
\begin{align*}
\int_{|\epsilon| \leq 1/2} |S^{(1)}(\epsilon) S^{(2)}_{\ell \ell^*}(\epsilon)| d\epsilon &\leq \left( \int_{|\epsilon| \leq 1/2} |S^{(1)}(\epsilon)|^2 d\epsilon \right)^{1/2} \left( \int_{|\epsilon| \leq 1/2} |S^{(2)}_{\ell \ell^*}(\epsilon)|^2 d\epsilon \right)^{1/2}, \\
&\ll X^{1-\eta/2} (\ln{X})^{O(1)}. 
\end{align*}
We bound the sum over $\ell, \ell^*$ trivially. This gives
\begin{equation*} M_{12} \ll X^{1-\eta/2} (\ln{X})^{O(1)} \sum_{q \leq P} \frac{(q,N) (d,q) \tau(q)^2}{dq}. \end{equation*}
By the inequality $(m,n) \leq \sum_{r|m,n} r$, we have
\begin{equation} \sum_{N \leq X} (q,N) \ll X \tau(q), \quad \sum_{d \leq D} \frac{(d,q)}{d} \ll \tau(q) \ln{X}. \label{dNq} \end{equation}
This gives
\begin{align}
\sum_{N \leq X} \sum_{d \leq D} \sup_{(d,l(N-l))=1} |M_{12}| &\ll X^{2-\eta/2} (\ln{X})^{O(1)} \sum_{q \leq P} \frac{\tau(q)^4}{q} \notag \\
&\ll X^{2-\eta/2} (\ln{X})^{O(1)} \label{SM12b}
\end{align}

We write
\[ S^{(1)}(\epsilon;\Lambda') S^{(2)}(\epsilon;\beta) = S^{(1)}(\epsilon;1) S^{(2)}(\epsilon;\beta) + S^{(1)}(\epsilon;\Lambda'-1) S^{(2)}(\epsilon;\beta) \]
and denote by $M_{111}$ and $M_{112}$ the contributions from $1 \times \beta$ and $(\Lambda'-1) \times \beta$ to $M_{11}$, respectively. 

Let
\[ M_{01} = \sum_{q \leq P} \frac{\mu(q)^2 \varphi((d,q))}{\varphi(d) \varphi(q)^2} c_{q/(d,q)}(N) \int_{\mathbb{T}} S^{(1)}(\alpha;1) S^{(2)}(\alpha;\beta) e(-N\alpha) d\alpha \]
\[ M_{02} = \sum_{q > P} \frac{\mu(q)^2 \varphi((d,q))}{\varphi(d) \varphi(q)^2} c_{q/(d,q)}(N) \int_{\mathbb{T}} S^{(1)}(\alpha;1) S^{(2)}(\alpha;\beta) e(-N\alpha) d\alpha \]
so $M_0=M_{01}+M_{02}$ by Lemma \ref{sing}. Then, by \eqref{dNq}, we have
\begin{align*}
\sum_{N \leq X} \sum_{d \leq D} \sup_{(d,l(N-l))=1} |M_{02}| &\ll X \sum_{N \leq X} \sum_{d \leq D} \sum_{q > P} \frac{(q,N) (d,q)}{dq^2} \ln{q}, \\
&\ll X^2 (\ln{X})^{O(1)} \sum_{q > P} \frac{\tau(q)^2}{q^2} \ln{q} \ll X^2 P^{-1} (\ln{X})^{O(1)}.
\end{align*}

By definition, one has
\[ M_{01} - M_{111} = \sum_{q \leq P} \frac{\mu(q)^2 \varphi((d,q))}{\varphi(d) \varphi(q)^2} c_{q/(d,q)}(N) \int_{P/(qX)<|\epsilon| \leq 1/2} S^{(1)}(\epsilon;1) S^{(2)}(\epsilon;\beta) e(-N\epsilon) d\epsilon. \]
This, together with \eqref{dNq} and $(q,N) \leq \sum_{r|N,q} r$, gives
\begin{align}
\sum_{N \leq X} \sum_{d \leq D} \sup_{(d,l(N-l))=1} |M_{111}-M_{01}| &\ll (\ln{X})^{O(1)} \sum_{N \leq X} \sum_{q \leq P} \frac{(q,N) \tau(q)^3}{q^2} \notag \\
& \times \left| \int_{P/(qX)<|\epsilon| \leq 1/2} S^{(1)}(\epsilon;1) S^{(2)}(\epsilon;\beta) e(-N\epsilon) d\epsilon \right|, \notag \\
&\ll (\ln{X})^{O(1)} \sum_{r \leq P} r \sum_{\substack{q \leq P \\ r|q}} \frac{\tau(q)^3}{q^2} \notag \\
& \times \sum_{\substack{N \leq X \\ r|N}} \left| \int_{P/(qX)<|\epsilon| \leq 1/2} S^{(1)}(\epsilon;1) S^{(2)}(\epsilon;\beta) e(-N\epsilon) d\epsilon \right|. \label{SM111bt}
\end{align}
Using the Cauchy--Schwarz inequality and Bessel's inequality, we see that the last sum is bounded by
\[ \ll \left( \frac{X}{r} \right)^{1/2} \left( \int_{P/(qX)<|\epsilon| \leq 1/2} |S^{(1)}(\epsilon;1) S^{(2)}(\epsilon;\beta)|^2 d\epsilon \right)^{1/2}. \]
By Lemma \ref{sumw}, we have
\begin{align*}
\int_{P/(qX)<|\epsilon| \leq 1/2} |S^{(1)}(\epsilon;1) S^{(2)}(\epsilon;\beta)|^2 d\epsilon &\ll \left( \frac{qX}{P} \right)^2 \int_{|\epsilon| \leq 1/2} |S^{(2)}(\epsilon;\beta)|^2 d\epsilon, \\
&\ll \left( \frac{qX}{P} \right)^2  X (\ln{X})^{O(1)}.
\end{align*}
Inserting these bounds into \eqref{SM111bt}, we obtain
\begin{align}
\sum_{N \leq X} \sum_{d \leq D} \sup_{(d,l(N-l))=1} |M_{111}-M_{01}| &\ll (\ln{X})^{O(1)} \sum_{r \leq P} r \sum_{\substack{q \leq P \\ r|q}} \frac{\tau(q)^3}{q^2} \times \left( \frac{X}{r} \right)^{1/2} \left( \frac{qX}{P} \right) X^{1/2}, \notag \\
&\ll X^2 P^{-1} (\ln{X})^{O(1)} \sum_{r \leq P} \frac{\tau(r)^3}{r^{1/2}} \ll X^2 P^{-1/2} (\ln{X})^{O(1)}. \label{SM111b}
\end{align}

We use the Cauchy--Schwarz inequality to bound $M_{112}$. This gives
\begin{equation*} |M_{112}| \ll X^{1/2} \ln{X} \sum_{q \leq P} \frac{(q,N) (d,q)}{d q^2} \left( \int_{|\epsilon| \leq P/(qX)} |S^{(1)}(\epsilon;\Lambda'-1)|^2 d\epsilon \right)^{1/2}. \end{equation*}
Thus, 
\begin{align}
\sum_{N \leq X} \sum_{d \leq D} \sup_{(d,l(N-l))=1} |M_{112}| &\ll X^{3/2} (\ln{X})^2 \sum_{q \leq P} \frac{\tau(q)^2}{q^2} \left( \int_{|\epsilon| \leq P/(qX)} |S^{(1)}(\epsilon;\Lambda'-1)|^2 d\epsilon \right)^{1/2} \label{SM1123b}. 
\end{align}
We divide the sum over $q$ into dyadic ranges $Q < q \leq 2Q$, and use divisor sum bounds. This, together with \eqref{SM12b} and \eqref{SM111b}, gives \eqref{11-0}.

\vskip\baselineskip

For $M_{2}$, we divide the sum over $q$ into dyadic ranges $q \sim Q$. By Lemma \ref{R}, we have
\begin{align*}
& \sum_{N \leq X} \sum_{d \leq D} \sup_{(d,l(N-l))=1} \sum_{\substack{q \sim Q \\ r_1|[d,q] \\ \ell \ell^* r_2 | q}} \frac{|R(\xi_1,\xi_2,q,d,l,N)|r^{1/2}_2}{\varphi([d,q]) \varphi(q/\ell)} \\
&\ll X (\ln{X}) r^{1/2}_2 \ell \tau(r_2) \underset{\substack{r_1|[d,q] \\ \ell \ell^* r_2|q}}{\sum_{d \leq D} \sum_{q \sim Q}} \frac{1}{[d,q]} \tau(q) \\
&\ll X (\ln{X}) r^{1/2}_2 \ell \tau(r_2) \1_{\substack{r_1 \ll QD \\ \ell \ell^* r_2 \ll Q}} \sum_{\substack{d,q \leq X \\ r_1|[d,q] \\ \ell \ell^* r_2|[d,q]}} \frac{1}{[d,q]} \tau([d,q]) \\
&\ll \1_{\substack{r_1 \ll QD \\ \ell \ell^* r_2 \ll Q}} \frac{\ell r^{1/2}_2 \tau(r_1)^3 \tau(r_2)^4 \tau(\ell \ell^*)^3}{[r_1,\ell \ell^* r_2]} X (\ln{X})^{O(1)} \\
\end{align*}
for each $r_1,r_2,\ell,\ell^*$. This, together with $\sum_{\ell \ell^*=h} \ell \leq h \tau(h)$ and Lemma \ref{tau} yields
\[ \sum_{N \leq X} \sum_{d \leq D} \sup_{(d,l(N-l))=1} |M_{2}| \ll X (\ln{X})^{O(1)} \sup_{Q \ll P} E_2. \]
This and \eqref{11-0} complete the proof. 

\end{proof}

\begin{lem}\label{E12}
If $\theta \leq 9/40-2\varepsilon$, then for any $A>0$ and uniformly for all $Q \ll P$, we have
\begin{equation} \label{E1b} E_1 \ll \frac{X^{1/2}}{(\ln{X})^A} \end{equation}
\begin{equation} \label{E2b} E_2 \ll \frac{X}{(\ln{X})^A}. \end{equation}
\end{lem}

\begin{proof}
Recall that
\[ E_1 = Q^{-1} \left( \int_{|\epsilon| \leq P/(QX)} |S^{(1)}(\epsilon;\Lambda'-1)|^2 d\epsilon \right)^{1/2} \]
\begin{align*}
E_2 &= \underset{r_1r_2 \neq 1}{\sum_{r_1 \leq QD} \sum_{h r_2 \leq Q}} \sideset{}{^*}\sum_{\xi_1 \bmod r_1} \sideset{}{^*}\sum_{\xi_2 \bmod r_2} \frac{(r_1,h r_2)}{r_1 r^{1/2}_2} (\tau(r_1) \tau(r_2) \tau(h))^{4} \\
& \times \int_{|\epsilon| \leq P/(QX)} |S^{(1)}(\epsilon;\xi_1) S^{(2)}_{h}(\epsilon;\xi_2)| d\epsilon
\end{align*}

We enlarge the range of integration to $[-P/X,P/X]$ and use Lemma \ref{Gal} \eqref{Gal1}. This gives
\[ E_1 \ll X^{1/2} \sup_{\substack{I \subset [\eta X,X] \\ |I| \leq 2XP^{-1}}} \frac{1}{XP^{-1}} \left| \sum_{n \in I} (\Lambda'(n)-1) \right|. \]
We use Lemma \ref{D21} if $|I| \geq X^{1-\varepsilon} P^{-1} \geq X^{7/12+\varepsilon}$ and use trivial bounds otherwise. This gives \eqref{E1b}.

We write
\[ I^{(1)}(r) = \sideset{}{^*}\sum_{\chi \bmod r} \int_{-X^{\varepsilon} PQ^{-1}}^{X^{\varepsilon} PQ^{-1}} \left| \sum_{n \leq X} \Lambda'(n) \chi(n) n^{it} \right| dt, \]
\[ I^{(2)}_h(r) = \sideset{}{^*}\sum_{\chi \bmod r} \int_{-X^{\varepsilon} PQ^{-1}}^{X^{\varepsilon} PQ^{-1}} \left| \sum_{n \leq X/h} \beta(hn) \chi(n) n^{it} \right| dt. \]
We have 
\[ \sideset{}{^*}\sum_{\chi \bmod r} \left( \int_{|\epsilon| \leq P/(QX)} |S^{(1)}(\epsilon;\chi)|^2 d\epsilon \right)^{1/2} \ll X^{-1/2} I^{(1)}(r) \]
by Lemma \ref{Gal}, and setting $h\epsilon=\epsilon'$, 
\[ \sideset{}{^*}\sum_{\chi \bmod r} \left( \int_{|\epsilon| \leq P/(QX)} |S^{(2)}_h(\epsilon;\chi)|^2 d\epsilon \right)^{1/2} \ll X^{-1/2} I^{(2)}_h(r). \]
Thus, by the Cauchy--Schwarz inequality, it remains to show that
\begin{equation} \label{R12} E'_2 \coloneqq R^{-1+o(1)}_1 R^{-1/2+o(1)}_2 \underset{r_1r_2 \neq 1}{\sum_{r_1 \sim R_1} \sum_{r_2 \sim R_2}} \sum_{h \ll Q/R_2} (r_1,hr_2) \tau(h)^4 I^{(1)}(r_1) I^{(2)}_h(r_2) \ll \frac{X^2}{(\ln{X})^{A+O(1)}} \end{equation}
for all $A>0$ and $Q \ll P, R_1 \ll QD, R_2 \ll Q$.

Note that for any $C>0$, Lemmas \ref{D22} and \ref{D1} give
\begin{equation} \label{I1log} I^{(1)}(r) \ll 
\begin{cases}
X (\ln{X})^{-B} & (1<r \leq (\ln{X})^{C}), \\
X (\ln{X})^{O(1)} & (r=1), \\
\end{cases}
\end{equation}
\begin{equation} \label{I2log} I^{(2)}_h(r) \ll 
\begin{cases}
Xh^{-1} (\ln{X})^{-B} & (1<r \leq (\ln{X})^{C}), \\
Xh^{-1} (\ln{X})^{O(1)} & (r=1), \\
\end{cases}
\end{equation}
for any $B>0$.

We use 
\[ (r_1,h r_2) \leq \sum_{\substack{g|r_1,h r_2 \\ g \ll \min(R_1,Q)}} g, \]
and
\[ \sum_{\substack{r_1 \sim R_1 \\ g|r_1}} I^{(1)}(r_1) \ll \left( X+ X^{11/20+\varepsilon} PQ^{-1} R^2_1 g^{-1} \right). \]
The second estimate follows by bounding the contribution of $\Lambda-\Lambda'$ trivially, and then applying Lemma \ref{D1} to the sum with $\Lambda$. 
This shows 
\begin{align*}
E'_2 &\ll R^{-1+o(1)}_1 R^{-1/2+o(1)}_2 (X\min(R_1,Q)+X^{11/20+\varepsilon} PQ^{-1} R^2_1) \\
&\times \sum_{h \ll Q/R_2} \tau(h)^5 \sum_{r_2 \sim R_2} I^{(2)}_h(r_2) 
\end{align*}
By Lemma \ref{D1}, we have
\begin{align*}
\sum_{h \ll Q/R_2} \tau(h)^5 \sum_{r_2 \sim R_2} I^{(2)}_h(r_2) &\ll \sum_{h \ll Q/R_2} \tau(h)^5 \left( \frac{X}{h}+\left(\frac{X}{h}\right)^{11/20+\varepsilon} PQ^{-1} R^2_2 \right)  \\
&\ll \left( X+X^{11/20+\varepsilon} PQ^{-1} R^2_2 (Q/R_2)^{9/20+\varepsilon} \right) (\ln{X})^{O(1)}. 
\end{align*}
Since
\[ R^{-1+o(1)}_1 (X\min(R_1,Q)+X^{11/20+\varepsilon} PQ^{-1} R^2_1) \ll X(Q^{\varepsilon/10} \min(1,Q/R_1) + X^{-9/20+2\varepsilon} PD), \]
\[ R^{-1/2+o(1)}_2 \left( X+X^{11/20+\varepsilon} PQ^{-1} R^2_2 (Q/R_2)^{9/20+\varepsilon} \right) \ll X (R^{-1/2+o(1)}_2 + X^{-9/20+2\varepsilon} P^{3/2}), \]
and $P \leq X^{9/40-\varepsilon}$, $PD \leq X^{9/20-3\varepsilon}$, we have \eqref{R12} for the cases $R_1 \geq Q^{1+\varepsilon} (\ln{X})^{10A}$ or $R_2 \geq Q^{\varepsilon} (\ln{X})^{10A}$.

Suppose that $R_1 \leq Q^{1+\varepsilon} (\ln{X})^{10A}$ and $R_2 \leq Q^{\varepsilon} (\ln{X})^{10A}$. 
We use 
\[ (r_1,h r_2) \leq (r_1,h) \sum_{\substack{g|r_1,r_2 \\ g \ll \min(R_1,R_2)}} g \] and
\[ \sum_{\substack{r_2 \sim R_2 \\ g|r_2}} I^{(2)}_h(r_2) \ll \left( \frac{X}{h}+\left(\frac{X}{h}\right)^{11/20+\varepsilon} \frac{PQ^{-1} R^2_2}{g} \right). \]
This shows 
\begin{align*}
E'_2 &\ll R^{-1+o(1)}_1 R^{-1/2+o(1)}_2 \sum_{r_1 \sim R_1} I^{(1)}(r_1) \sum_{h \ll Q/R_2} (r_1,h) \tau(h)^4 \left( \frac{X}{h} \min(R_1,R_2) +\left(\frac{X}{h}\right)^{11/20+\varepsilon} PQ^{-1} R^2_2 \right) \\
&\ll R^{-1+o(1)}_1 R^{-1/2+o(1)}_2 (X+X^{11/20+\varepsilon} PQ^{-1} R^2_1) \\
&\quad \times \left(X\min(R_1,R_2)+X^{11/20+\varepsilon} PQ^{-1} R^2_2 (Q/R_2)^{9/20+\varepsilon} \right) (\ln{X})^{O(1)}.
\end{align*}
Noting that, say,
\[ R^{-1+o(1)}_1 R^{-1/2+o(1)}_2 \min(R_1,R_2) \ll R^{-1/10}_1 R^{-1/10}_2 \]
we have \eqref{R12} if $R_1 \geq (\ln{X})^{100A}$ or $R_2 \geq (\ln{X})^{100A}$.

It remains to consider the cases where $R_1, R_2 \leq (\ln{X})^{100A}$. 
In this case, we have
\begin{align*} E'_2 &\ll (\ln{X})^{O_{A}(1)} \sup_{\substack{1<r_1 r_2 \leq (\ln{X})^{200A}}} \sum_{h \ll Q} \tau(h)^4 I^{(1)}(r_1) I^{(2)}_h(r_2),
\end{align*}
This, together with \eqref{I1log} and \eqref{I2log}, yields \eqref{R12}. 
\end{proof}

We now prove Theorem \ref{T1p}. 

\begin{proof}[Proof of Theorem \ref{T1p}]
Let $A>0$ and $\theta \in (0,1/6)$. By \eqref{Morth}, we have
\[ \sum_{d \leq X^{\theta}} \sup_{l:(d,l(N-l))=1} \left| \sum_{\substack{m+n=N \\ m \equiv l \bmod d}} \Lambda(m) \beta(n) w_1(m/X) w_2(n/X) - M_0 \right| \ll E_{\mathfrak{M}} + E_{\mathfrak{m}} + \frac{X}{(\ln{X})^A} \]
where $E_{\mathfrak{M}}=E_{\mathfrak{M}}(N)$ and $E_{\mathfrak{m}}=E_{\mathfrak{m}}(N)$ are defined by
\[ E_{\mathfrak{M}} = \sum_{d \leq X^{\theta}} \sup_{l:(d,l(N-l))=1} \left| \int_{\mathfrak{M}} S^{(1)}_{l,d}(\alpha) S^{(2)}(\alpha) e(-N\alpha) d\alpha - M_0 \right|, \]
\[ E_{\mathfrak{m}} = \sum_{d \leq X^{\theta}} \sup_{l} \left| \int_{\mathfrak{m}} S^{(1)}_{l,d}(\alpha) S^{(2)}(\alpha) e(-N\alpha) d\alpha \right|. \]
By Lemmas \ref{Mred} and \ref{E12}, we have
\[ \sum_{X/2<N\leq X} E_{\mathfrak{M}} \leq \frac{X^2}{(\ln{X})^{2A}}, \]
and by Lemma \ref{Min}
\[ \sum_{X/2<N\leq X} E^2_{\mathfrak{m}} \leq \frac{X^3}{(\ln{X})^{3A}}. \]
Since
\[ \1_{E_{\mathfrak{M}} \geq X(\ln{X})^{-A}} \leq X^{-1} (\ln{X})^{A} E_{\mathfrak{M}}, \quad \1_{E_{\mathfrak{m}} \geq X(\ln{X})^{-A}} \leq X^{-2} (\ln{X})^{2A} E^2_{\mathfrak{m}} \]
for $N \in (X/2,X]$, we see that
\begin{align*}
&\#\{ N \in (X/2,X] : E_{\mathfrak{M}} + E_{\mathfrak{m}} \geq 2X(\ln{X})^{-A} \} \\
&\leq \sum_{X/2<N \leq X} (\1_{E_{\mathfrak{M}} \geq X(\ln{X})^{-A}}+\1_{E_{\mathfrak{m}} \geq X(\ln{X})^{-A}}) \\
&\leq X^{-1} (\ln{X})^{A} \sum_{X/2<N \leq X} E_{\mathfrak{M}} + X^{-2} (\ln{X})^{2A} \sum_{X/2<N \leq X} E^2_{\mathfrak{m}} \ll \frac{X}{(\ln{X})^A}. \\
\end{align*}
This gives the first statement. 

Since the proof of the second statement is similar, we only describe the differences. 
By orthogonality, we have
\[ \sum_{\substack{2m+n=N \\ m \equiv l \bmod d}} \Lambda'(m) \beta(n) w_1(m/X) w_2(n/X) = \int_{\mathfrak{M} \cup \mathfrak{m}} S^{(1)}_{l,d}(2\alpha) S^{(2)}(\alpha) e(-N\alpha) d\alpha. \]
The contributions from the minor arc can be treated as before. Thus it suffices to evaluate
\[ M' \coloneqq \int_{\mathfrak{M}} S^{(1)}_{l,d}(2\alpha) S^{(2)}(\alpha) e(-N\alpha) d\alpha, \]
for odd $N \in (X/2,X]$. For $q\geq 1$, write $q=q^{(1)}q^{(2)}$ with $q^{(1)}=(q,2^{\infty})$ so $2 \nmid q^{(2)}$ and $(q^{(1)},q^{(2)})=1$. 
Then
\begin{align*}
\sum_{a \in \pZ{q}} \1_{m \equiv l \bmod d} e(a(2m+n-N)/q) &= \left( \sum_{a^{(1)} \in \pZ{q^{(1)}}} e(a^{(1)}(2m+n-N)/q^{(1)}) \right) \\
&\quad \times \left( \sum_{a^{(2)} \in \pZ{q^{(2)}}} \1_{m \equiv l \bmod d} e(a^{(2)}(2m+n-N)/q^{(2)}) \right) 
\end{align*}
for any $m,n$ and odd $d$, integer $l$ with $(d,l)=(d,N-2l)=1$. 
We use a decomposition similar to that in Lemma \ref{sumA}. In the sum over $a^{(2)}$, we make the change of variables $2b_1\mapsto b_1$. This gives
\begin{align*}
&\sum_{a \in \pZ{q}} e(-Na/q) S^{(1)}_{l,d}(2\epsilon+2a/q) S^{(2)}(\epsilon+a/q) \\
&= \sum_{\ell \ell^*|q^{(2)}} \frac{\mu(\ell^*)}{\varphi(q^{(1)})^2 \varphi([d,q^{(2)}]) \varphi(q^{(2)}/\ell)} \sum_{\chi^{(1)}_1 \bmod q^{(1)}} \sum_{\chi^{(1)}_2 \bmod q^{(1)}} \chi^{(1)}_2(\ell \ell^*) \tau_{-2}(\overline{\chi^{(1)}_1}) \tau(\overline{\chi^{(1)}_2}) \tau_N(\chi^{(1)}_1 \chi^{(1)}_2) \\
& \times \sum_{\chi^{(2)}_1 \bmod [d,q^{(2)}]} \sum_{\chi^{(2)}_2 \bmod q^{(2)}/\ell} \chi^{(2)}_1(2) (\chi^{(2)}_2)^*(\ell^*) \tau(\overline{\chi^{(2)}_2}) R(\chi^{(2)}_1,\chi^{(2)}_2,q^{(2)},d,2l,N) \\
&\times S^{(1)}(\epsilon;(\chi^{(1)}_1 \chi^{(2)}_1)^*) S^{(2)}_{\ell \ell^*}(\epsilon;(\chi^{(1)}_2 \chi^{(2)}_2)^*).
\end{align*}

Let $M'_1, M'_2$ be the contributions from the term with $(\chi^{(1)}_1 \chi^{(2)}_1)^*=(\chi^{(1)}_2 \chi^{(2)}_2)^*=1$, $(\chi^{(1)}_1 \chi^{(2)}_1)^*\neq 1$ or $(\chi^{(1)}_2 \chi^{(2)}_2)^*\neq 1$ to $M'$, respectively. 

We denote the conductor of $\chi^{(j)}_i$ by $r^{(j)}_{i}$. Then Lemmas \ref{R} and \ref{tau} give
\[ |\tau_{-2}(\overline{\chi^{(1)}_1}) \tau(\overline{\chi^{(1)}_2}) \tau_N(\chi^{(1)}_1 \chi^{(1)}_2) \tau(\overline{\chi^{(2)}_2}) R(\chi^{(2)}_1,\chi^{(2)}_2;q^{(2)},d,2l,N)| \ll q \tau(q) (N,q)^{1/2} (r^{(1)}_2 r^{(2)}_2)^{1/2} \ln{X} \]
Thus, $M'_2$ can be treated as before. 

If $(\chi^{(1)}_1 \chi^{(2)}_1)^*=(\chi^{(1)}_2 \chi^{(2)}_2)^*=1$, then
\[ \tau_{-2}(\overline{\chi^{(1)}_1}) \tau(\overline{\chi^{(1)}_2}) \tau_N(\chi^{(1)}_1 \chi^{(1)}_2) = c_{q^{(1)}}(2) \mu(q^{(1)}) c_{q^{(1)}}(N) \]
\[ \tau(\overline{\chi^{(2)}_2}) R(\chi^{(2)}_1,\chi^{(2)}_2;q^{(2)},d,2l,N) = \mu(q^{(2)}/\ell) \mu(q^{(2)}) c_{q^{(2)}/(d,q^{(2)})}(N) \]
Noting that $c_{q^{(1)}}(2) \mu(q^{(1)}) = \mu((q,2)) \mu(q^{(1)})^2$ and $(d,q^{(2)}) = (d,q)$, we therefore have 
\[ M'_1 = \sum_{q \leq P} \sum_{\ell \ell^*|q^{(2)}} \frac{\mu(q)^2 \mu((q,2)) \mu(\ell \ell^*) \varphi((d,q)) \varphi(\ell)}{\varphi(d) \varphi(q)^2} c_{q/(d,q)}(N) \int_{|\epsilon| \leq P/(qX)} S^{(1)}(2\epsilon) S^{(2)}_{\ell \ell^*}(\epsilon) e(-N\epsilon) d\epsilon. \]
The contributions from $\ell \ell^*>1$ can be treated as before. 
We write $S^{(1)}(2\epsilon)=(S^{(1)}(2\epsilon)-\tfrac{1}{2}T(\epsilon))+\tfrac{1}{2}T(\epsilon)$, where
\[ T(\epsilon) = \sum_{n} e(\epsilon n) w_1(n/(2X)). \]
The contributions from $S^{(1)}(2\epsilon)-\tfrac{1}{2}T(\epsilon)$ can be treated as before. 
For the rest of the argument, we use 
\[ \sum_{q} \frac{\mu(q)^2 \mu((q,2)) \varphi((d,q))}{\varphi(d) \varphi(q)^2} c_{q/(d,q)}(N) = \frac{\mathfrak{S}'(N)}{\psi_N(d)} \]
in place of Lemma \ref{sing}. This gives the second statement of Theorem \ref{T1p}. 

\end{proof}

\section{Sieve}

We use Kuhn's weighted sieve with switching as in \cite{matomaki2025weighted} to prove Theorem \ref{T2}, and Richert's weighted sieve to prove Theorem \ref{T3}.
In both cases, we use a result from the linear sieve. 
In order to simplify the proof of Theorem \ref{T2}, we borrow the following result of Bantle \cite{bantle1986dependence}, which gives a sharper error term than the standard form. 

Let $\varpi(n,z)$ denote the indicator function of $z$-rough numbers. 
\begin{lem}\label{LS}
Let $\omega^{+}, \omega^{-}: \mathbb{R}_{>0} \to \mathbb{R}_{\geq 0}$ be the functions defined by
\[ \omega^{+}(s) = \frac{2}{s}, \quad \omega^{-}(s) = 0, \quad (s \leq 2) \] 
\[ (s\omega^{\pm}(s))' = \omega^{\mp}(s-1) \quad (s>2) \]

Let $\mathcal{A}=(a_n)$ be a finitely supported non-negative sequence and
\[ S(\mathcal{A},z) = \sum_{n} a_n \varpi(n,z). \] 
Let $g$ be a multiplicative function satisfying $0 \leq g(p) < 1$ for all $p$. 
Let $x,D,z\geq 2$ and write
\[ R_d(\mathcal{A}) = \sum_{d|n} a_n - xg(d). \]
Suppose that $L \geq 2$ and
\begin{equation} \prod_{w \leq p < y} (1 - g(p))^{-1} \leq \frac{\ln{y}}{\ln{w}} \left( 1 + \frac{L}{\ln{w}} \right) \label{dim1} \end{equation}
holds for all $2 \leq w<y\leq z$. 
Let $K \geq 1$ and assume that $s=\ln{D}/\ln{z} \in [1/K,K]$. 
Then we have
\[ x (\omega^-(s) - H) e^{\gamma} V(z) - \sum_{d \leq D} \mu(d)^2 |R_d(\mathcal{A})| \leq S(\mathcal{A},z) \leq x (\omega^+(s) + H) e^{\gamma} V(z) + \sum_{d \leq D} \mu(d)^2 |R_d(\mathcal{A})|, \]
where
\[ V(z) \coloneqq \prod_{p<z} (1-g(p)), \]
and
\[ H = C_K L^5 (\ln{z})^{-1/3} \]
for some constant $C_K>0$ depending only on $K$. 
\end{lem}

\begin{proof}
This follows from \cite[Theorem 1]{bantle1986dependence}. 
\end{proof}

By Mertens' theorem, one has
\begin{equation} e^{\gamma} V(z) = \frac{1+o(1)}{\ln{z}} \prod_{p<z} (1-g(p)) (1-1/p)^{-1} \label{V_Mer}. \end{equation}

\subsection{Primes satisfying $p_1+p_2=N, p_1-p_2+1 \in \mathbb{P}_k$}

Let $w$ be a smooth function supported on $[1/10,2/10]$ and positive on $(1/10,2/10)$. 
For arithmetic functions $f_i$, we write
\[ \Sigma(f_1,f_2,f_3) = \sum_{m+n=N} f_1(m) f_2(n) f_3(m-n+1) w(n/X). \]
Note that $m/X=(N-n)/X \in [3/10,9/10]$ and $m-n+1>0$ whenever $N \in (X/2,X]$ and $w(n/X) \neq 0$.

Let $k \geq 2$ be an integer, and let $\theta>0$. 
Let $F$ be a piecewise smooth non-negative function supported on $[0,\theta-\varepsilon]$ for some $\varepsilon>0$, and define $W$ and $W_{>k}$ by
\[ W(n) = 1-\sum_{p|n} F\left(\frac{\ln{p}}{\ln{X}}\right), \]
\[ W_{>k}(n) = 
\begin{cases}
0 & (n \in \mathbb{P}_k)  \\
\max(0,W(n)) & (n \notin \mathbb{P}_k) \\
\end{cases} \]
so that
\[ \1_{\mathbb{P}_k}(n) \geq W(n)-W_{>k}(n). \]

Let $s>2$ and put $z=X^{1/s}$. Let
\[ S_1 = \Sigma(\Lambda',\Lambda',W \varpi(\cdot,z)), \quad S_2 = \Sigma(\Lambda',\Lambda',W_{>k}\varpi(\cdot,z)). \]
The proof of Theorem \ref{T2} is reduced to showing
\[ S_1 - S_2 > 0 \]
for almost all even $N$, $k=4$ and some $s,F$.

For an even integer $N>0$, let $\mathfrak{T}(N)=4 \prod_{p>2} C_p$ with
\begin{align*}
C_p =
\begin{cases}
\left(1 - \frac{1}{(p-1)^2} \right) \left(1- \frac{1}{p} \right)^{-1} & p|N(N^2-1) \\
\left(1 - \frac{1}{(p-1)^2} \right) \left(1 - \frac{1}{p-2} \right) \left(1- \frac{1}{p} \right)^{-1} & p\nmid N(N^2-1) \\
\end{cases}
\end{align*}
Let
\[ c_w = \int_{\mathbb{R}} w(u) du. \]

\begin{lem}\label{t1S1}
Let $A>0$. If $1/s<\theta < 1/6$, then we have
\[ S_1 \geq \frac{c_w X \mathfrak{T}(N)}{\ln{z}} \left(\omega^{-}(s\theta) - \int_{1/s}^{\theta} \frac{F(u) \omega^{+}(s(\theta-u))}{u} du +o(1) \right) \]
for all even integers $N \in (X/2,X]$ with $O(X(\ln{X})^{-A})$ exceptions. 
\end{lem}

\begin{proof}
Discarding the contribution from prime powers trivially, we see that
\[ S_1 = \Sigma(\Lambda,\Lambda,W\varpi(\cdot,z)) + O(X^{1/2+\varepsilon}). \]

Let
\[ S_{11} = \Sigma(\Lambda,\Lambda,\varpi(\cdot,z)), \quad S_{12} = \sum_{p>z} F\left(\frac{\ln{p}}{\ln{X}}\right) \Sigma(\Lambda,\Lambda,\1_{p|\cdot} \varpi(\cdot,z)), \]
so $S_1=S_{11}-S_{12}+O(X^{1/2+\varepsilon})$. 

We define $\mathcal{A}=(a_n)$ by
\[ a_n = \Sigma(\Lambda,\Lambda,\1_{\cdot=n}). \]
Then we see that $a_n \geq 0$ and
\[ S_{11} = S(\mathcal{A},z) . \]
We apply Lemma \ref{LS} with 
\begin{align*} g(d) = 
\begin{dcases}
\frac{1}{\psi_N(d)} & (d,2(N^2-1)) = 1 \\
0 & (d,2(N^2-1)) \neq 1
\end{dcases}
\end{align*}
\[ x = Y \mathfrak{S}(N), \quad D=X^{\theta}, \]
where $Y= \sum_{n} w(n/X)$. This gives
\begin{equation} S_{11} \geq Y \mathfrak{S}(N) e^{\gamma} V(z) (\omega^{-}(s \theta)+o(1)) + O\left( \sum_{d \leq D} \left| R_d \right| \right), \label{S11p} \end{equation}
where
\[ R_d = \Sigma(\Lambda,\Lambda,\1_{d|\cdot}) - Y\mathfrak{S}(N) g(d). \]
Using \eqref{V_Mer} and 
\[ \prod_{\substack{p|N(N^2-1) \\ p>z}} \left(1 + O\left( \frac{1}{p} \right) \right) = \exp\left( O\left( \sum_{\substack{p|N(N^2-1) \\ p>z}} \frac{1}{p} \right) \right) = 1+o(1), \]
one has
\begin{align*}
e^{\gamma} V(z) &= \frac{1+o(1)}{\ln{z}} \prod_{p < z} \left(1 - \frac{\1_{p\nmid 2(N^2-1)}}{\psi_N(p)} \right) \left(1- \frac{1}{p} \right)^{-1} \\
&= \frac{2+o(1)}{\ln{z}} \prod_{p>2} A_p, \quad \text{where } A_p=
\begin{cases}
\left(1 - \frac{1}{p-1} \right) \left(1- \frac{1}{p} \right)^{-1} & p|N \\
\left(1- \frac{1}{p} \right)^{-1} & p|N^2-1 \\
\left(1 - \frac{1}{p-2} \right) \left(1- \frac{1}{p} \right)^{-1} & p\nmid N(N^2-1) \\
\end{cases}
\end{align*}
Since
\begin{align*}
\mathfrak{S}(N) &= 2 \prod_{p>2} B_p, \quad \text{where }
B_p =
\begin{cases}
\left(1 + \frac{1}{p-1} \right) & p|N \\
\left(1 - \frac{1}{(p-1)^2} \right) & p|N^2-1 \\
\left(1 - \frac{1}{(p-1)^2} \right) & p\nmid N(N^2-1) \\
\end{cases}
\end{align*}
and $C_p=A_p B_p$, we obtain
\[ S_{11} \geq \frac{Y \mathfrak{T}(N)}{\ln{z}} (\omega^{-}(s\theta)+o(1)) + O\left( \sum_{d \leq D} \left| R_d \right| \right). \]

For $S_{12}$, we apply Lemma \ref{LS} with
\[ \mathcal{A}_p = (\Sigma(\Lambda,\Lambda,\1_{\cdot=pn}))_n, \quad x = \frac{Y \mathfrak{S}(N)}{\psi_N(p)}, \quad D=X^{\theta}/p, \]
for each $p$. 
This and $g(p) = 1/p + O(1/p^2 + \1_{p|2(N^2-1)}/p)$ give
\[ S_{12} \leq \sum_{z<p} \frac{Y \mathfrak{T}(N)}{p\ln{z}} \left( F\left(\frac{\ln{p}}{\ln{X}}\right) \omega^{+}\left(\frac{\ln{X^{\theta}/p}}{\ln{z}}\right) +o(1) \right) + O\left( \frac{X^{1+\varepsilon}}{z} + \sum_{p>z} \sum_{d \leq X^{\theta}/p} \left| R_{dp} \right| \right) \]
Since
\[ \sum_{z<p} \frac{1}{p} F\left(\frac{\ln{p}}{\ln{X}}\right) \omega^{+}\left(\frac{\ln{X^{\theta}/p}}{\ln{z}}\right) = \int_{1/s}^{\theta} \frac{F(u) \omega^{+}(s(\theta-u))}{u} du + o(1), \]
and $\sum_{\substack{p|d' \\ p>z}} 1 = O(1)$, combining this with \eqref{S11p} gives
\[ S_1 \geq \frac{Y \mathfrak{T}(N)}{\ln{z}} \left(\omega^{-}(s\theta) - \int_{1/s}^{\theta} \frac{F(u) \omega^{+}(s(\theta-u))}{u} du +o(1) \right) + O\left( \sum_{d \leq X^{\theta}} \left| R_{d} \right| \right). \]
Note that $Y = c_{w}X(1+o(1))$. Thus, it remains to show that
\begin{equation} \sum_{d \leq X^{\theta}} \left| R_{d} \right| = o\left( \frac{X}{\ln{X}} \right), \label{Rd1} \end{equation}
for almost all $N$.

If $m+n=N$ with $2 \nmid m,n$ and $d|m-n+1$, then we see that $2\nmid d$ and $2m \equiv m+n-1 \equiv N-1 \bmod d$. Furthermore, if $m$ and $n$ are prime $\sim X$, and $d \leq X^{\theta}$, we have $(d,m)=(d,N-1)=1$ and $(d,n)=(d,N+1)=1$. 
Thus,
\[ |R_d| \ll\sup_{l:(d,l(N-l))=1} \left| \sum_{\substack{m+n=N \\ m \equiv l \bmod d}} \Lambda(m) \Lambda(n) w(n/X) - \frac{\mathfrak{S}(N) Y}{\psi_N(d)} \right| + X^{1/2+\varepsilon}. \]
Let $w_1$ be a smooth function supported on $[1/20,19/20]$ and $w_1(u)=1$ for $u \in [1/10,9/10]$. Then $w_1(m/X) w(n/X) = w(n/X)$ for all $m,n$ with $m+n \in (X/2,X]$. 
Therefore, \eqref{Rd1} follows from Theorem \ref{T1}.

\end{proof}

We define the function $G:\bigsqcup_{r=1}^{\infty} [0,1]^r \to \mathbb{R}_{\geq 0}$ by
\[ G(u_1,\ldots,u_r) = \max(0,1-F(u_1)-\cdots-F(u_r)). \]

\begin{lem}\label{t1S2}
Let $A>0$. If $\theta < 1/6$, then we have
\[ S_2 \leq \frac{c_w X \mathfrak{T}(N)}{\ln{z}} \left( \frac{2}{\theta} \sum_{r>k} \frac{1}{r!} \int_{\substack{u_i>1/s \\ u_1+\cdots+u_r=1}} \frac{G(u_1,\ldots,u_r)}{u_1 \cdots u_r} du +o(1) \right) \]
for all even integers $N \in (X/2,X]$ with $O(X(\ln{X})^{-A})$ exceptions. 
Here, the integration over $u_1+\cdots+u_r=1$ is understood as the integral over $u_1 + \cdots + u_{r-1} \leq 1$ with $u_r=1-u_1-\cdots-u_{r-1}$. 
\end{lem}

\begin{proof}

If $n>1$ is not squarefree and $\varpi(n,z)=1$, then exists $p>z$ with $p^2|n$. 
Thus
\[ \Sigma(\Lambda',\Lambda',W_{>k} \varpi(\cdot,z) \1_{\cdot \text{ is not squarefree}}) \ll (\ln{X})^2 \sum_{p>z} \Sigma(1,1,\1_{p^2|\cdot}) = o\left( \frac{X}{\ln{z}} \right). \]
Therefore, we see that
\[ S_2 = \sum_{r>k} S_{2,r} + o\left( \frac{X}{\ln{z}} \right), \]
where $S_{2,r} = \Sigma(\Lambda',\Lambda',W_{>k} \varpi(\cdot,z) \1_{\cdot \text{ is product of $r$ distinct primes}})$. 
Note that this sum is effectively finite, since the summand vanishes for $r>s$.

Let $r>k$ and $M>s$. We divide
\[ (0,1]^r = \bigsqcup_{1 \leq j_1,\ldots,j_r \leq M} I_{j_1} \times \cdots \times I_{j_r}, \quad I_{j} = ((j-1)/M, j/M]. \]
For each $1 < j_1,\ldots,j_r \leq M$, we define
\[ G_{j_1,\ldots,j_r} = \max_{u_i \in I_{j_i} (1\leq i \leq r)} \frac{\max(0,1-F(u_1)-\cdots-F(u_r)) \1_{u_i>1/s\forall i}}{u_1 \cdots u_r} \]
so
\[ W_{>k}(p_1 \cdots p_r) \varpi(p_1 \cdots p_r,z) \leq \sum_{1<j_1,\ldots,j_r \leq M} G_{j_1,\ldots,j_r} \prod_{i} \left( \frac{\ln{p_i}}{\ln{X}} \right) \1_{p_i \in X^{I_{j_i}}} \]
for all distinct primes $p_1, \ldots, p_r$ with $p_1 \cdots p_r \leq X$, where we write $X^{I_j}=\{n \in \mathbb{N}:\ln{n}/\ln{X} \in I_{j} \}$. Noting $\Lambda'(\cdot) \leq \Lambda(\cdot)$, we have
\begin{equation} S_{2,r} \leq \frac{1}{r!} \sum_{1<j_1,\ldots,j_r \leq M} G_{j_1,\ldots,j_r} \Sigma(\Lambda',\Lambda,\beta_{j_1,\ldots,j_r}) (\ln{X})^{-r} \label{S2rp} \end{equation}
where we denote
\[ \beta_{j_1,\ldots,j_r}(n) = \sum_{\substack{n_1 \cdots n_r=n \\ n_i \in X^{I_{j_i}} \forall i}} \Lambda(n_1) \cdots \Lambda(n_r). \]

We now fix $j_1,\ldots,j_r>1$ and define the sequence $\mathcal{A}=(a_n)_n$ by 
\[ a_n = \Sigma(\1_{\cdot=n},\Lambda,\beta_{j_1,\ldots,j_r}). \]
Then we see that
\[ \Sigma(\Lambda',\Lambda,\beta_{j_1,\ldots,j_r}) \leq S(\mathcal{A},X^{1/2}) \ln{X}. \]

We apply Lemma \ref{LS} with 
\begin{align*} g(d) = 
\begin{dcases}
\frac{1}{\psi_{N+1}(d)} & (d,N(N-1)) = 1, \\
0 & (d,N(N-1)) \neq 1,
\end{dcases}
\end{align*}
\[ x = Y_{j_1,\ldots,j_r} \mathfrak{S}'(N+1), \quad D=X^{\theta}, \]
where
\[ Y_{j_1,\ldots,j_r} = \frac{1}{2} \sum_{m+2n=N+1} \beta_{j_1,\ldots,j_r}(m) w(n/X). \]
This gives
\[ S(\mathcal{A},X^{1/2}) \leq Y_{j_1,\ldots,j_r} \mathfrak{S}'(N+1) e^{\gamma} V(X^{1/2}) (\omega^{+}(2\theta)+o(1)) + O\left( \sum_{d \leq D} \left| R_{d,j_1,\ldots,j_r} \right| \right), \]
where
\[ R_{d,j_1,\ldots,j_r} = \Sigma(\1_{d|\cdot},\Lambda,\beta_{j_1,\ldots,j_r}) - Y_{j_1,\ldots,j_r}\mathfrak{S}'(N+1) g(d). \]
Arguing as before, we have
\begin{align*}
e^{\gamma} V(X^{1/2}) &= \frac{2+o(1)}{\ln{X^{1/2}}} \prod_{2<p < X^{1/2}} \left(1 - \frac{\1_{p\nmid N(N-1)}}{\psi_{N+1}(p)} \right) \left(1- \frac{1}{p} \right)^{-1}, \\
&= \frac{2+o(1)}{\ln{X^{1/2}}} \prod_{p>2} A'_p, \quad \text{where } A'_p=
\begin{cases}
\left(1 - \frac{1}{p-1} \right) \left(1- \frac{1}{p} \right)^{-1} & p|N+1, \\
\left(1- \frac{1}{p} \right)^{-1} & p|N(N-1), \\
\left(1 - \frac{1}{p-2} \right) \left(1- \frac{1}{p} \right)^{-1} & p\nmid N(N^2-1). \\
\end{cases}
\end{align*}
Noting that
\begin{align*}
\mathfrak{S}'(N+1) &= 2 \prod_{p>2} B'_p, \quad \text{where }
B'_p =
\begin{cases}
\left(1 + \frac{1}{p-1} \right) & p|N+1, \\
\left(1 - \frac{1}{(p-1)^2} \right) & p|N(N-1), \\
\left(1 - \frac{1}{(p-1)^2} \right) & p\nmid N(N^2-1), \\
\end{cases}
\end{align*}
and $A'_p B'_p = C_p$, we therefore have 
\[ \mathfrak{S}'(N+1) e^{\gamma} V(X^{1/2}) (\omega^{+}(2\theta)+o(1)) = \frac{\mathfrak{T}(N)}{\ln{X}} (2/\theta+o(1)). \]

Inserting these into \eqref{S2rp}, we obtain
\begin{align}
S_{2,r} &\leq (2/\theta+o(1)) \sum_{1<j_1,\ldots,j_r \leq M} G_{j_1,\ldots,j_r} Y_{j_1,\ldots,j_r} (\ln{X})^{-r} \mathfrak{T}(N) \notag \\
&+ O\left( \sum_{1<j_1,\ldots,j_r \leq M} \sum_{d \leq X^{\theta}} |R_{d,j_1,\ldots,j_r}| \right). \label{S2r} 
\end{align}

We write $w^*_N(u)=w(((N+1)/X - u)/2)$ so that
\[ Y_{j_1,\ldots,j_r} = \frac{1}{2} \sum_{m} \beta_{j_1,\ldots,j_r}(m) w^*_N(m/X). \]
Thus, by the prime number theorem, we have
\[ Y_{j_1,\ldots,j_r} = \frac{1}{2} \sum_{\substack{n_1, \ldots, n_r \\ n_i \in X^{I_{j_i}}}} w^*_N(n_1 \cdots n_r/X) + O\left( \frac{X}{(\ln{X})^{r+100}} \right). \]
For each $n_i \in \mathbb{N}$, we have
\[ w^*_N(n_1 \cdots n_r/X) = \int \cdots \int_{y_i \in [n_i,n_i+1]} w^*_N(y_1 \cdots y_r/X) dy_1 \cdots dy_r + O\left( \frac{1}{n_1} + \cdots + \frac{1}{n_r} \right). \]
Thus,
\[ \sum_{\substack{n_1, \ldots, n_r \\ n_i \in X^{I_{j_i}}}} w^*_N(n_1 \cdots n_r/X) = \int \cdots \int_{y_i \in X^{I_{j_i}}} w^*_N(y_1 \cdots y_r/X) dy_1 \cdots dy_r + O(X^{1-1/M}). \]
Changing variables first to $v=y_1 \cdots y_r/X$ and then to $u_i=\ln{y_i}/\ln{X}$, we obtain
\begin{align*}
&= X \int_{\mathbb{R}} w^*_N(v) \left( \int_{\substack{y_i \in X^{I_{j_i}} \\ vX/y_1 \cdots y_{r-1} \in X^{I_{j_r}}}} \frac{dy_1 \cdots dy_{r-1}}{y_1 \cdots y_{r-1}} \right) dv \\
&= X (\ln{X})^{r-1} \int_{\mathbb{R}} w^*_N(v) \left( \int_{\substack{u_i \in I_{j_i} \\ 1-(u_1+\cdots+u_{r-1}) + \ln{v}/\ln{X} \in I_{j_r}}} du_1 \cdots du_{r-1} \right) dv \\
\end{align*}
Since
\[ \int_{\substack{u_i \in I_{j_i} \\ 1-(u_1+\cdots+u_{r-1}) + \ln{v}/\ln{X} \in I_{j_r}}} du_1 \cdots du_{r-1} = \int_{\substack{u_i \in I_{j_i} \\ u_1+\cdots+u_r=1}} du + o(1), \quad v \in [1/100,99/100], \]
and $\int_{\mathbb{R}} w^*_N(u) du = 2c_w$, we obtain
\[ Y_{j_1,\ldots,j_r} = c_w X (\ln{X})^{r-1} \left( \int_{\substack{u_i \in I_{j_i} \\ u_1+\cdots+u_r=1}} du + o(1) \right), \]
and thus
\begin{align*}
&\sum_{1<j_1,\ldots,j_r \leq M} G_{j_1,\ldots,j_r} Y_{j_1,\ldots,j_r} (\ln{X})^{-r} \\
&= c_w X (\ln{X})^{-1} \left( \int_{\substack{u_i \in [0,1] \\ u_1+\cdots+u_r=1}} \sum_{1<j_1,\ldots,j_r \leq M} G_{j_1,\ldots,j_r} \1_{u_i \in I_{j_i}} du + o(1) \right). 
\end{align*}
Since $F$ is piecewise smooth, we see that
\[ \left| \int_{\substack{u_i \in [0,1] \\ u_1+\cdots+u_r=1}} \sum_{1<j_1,\ldots,j_r \leq M} G_{j_1,\ldots,j_r} \1_{u_i \in I_{j_i}} du - \int_{\substack{u_i>1/s \\ u_1+\cdots+u_r=1}} \frac{G(u_1,\ldots,u_r)}{u_1 \cdots u_r} du \right| < \varepsilon \]
for any $\varepsilon>0$, by taking $M$ sufficiently large. 

It remains to treat the error term in \eqref{S2r}. 
Letting $m'=m-n+1$, we have
\begin{align*}
\Sigma(\1_{d|\cdot}, \Lambda, \beta_{j_1,\ldots,j_r}) &= \sum_{\substack{m+n=N \\ d|m}} \Lambda(n) \beta_{j_1,\ldots,j_r}(m-n+1) w(n/X), \\
&= \sum_{\substack{2n+m'=N+1 \\ n \equiv N \bmod d}} \Lambda(n) \beta_{j_1,\ldots,j_r}(m') w(n/X). 
\end{align*}
Arguing as before, we see that 
\begin{align*}
|R_{d,j_1,\ldots,j_r}| &\ll \sup_{l:(d,l(N+1-2l))=1} \left| \sum_{\substack{2n+m'=N+1 \\ n \equiv l \bmod d}} \Lambda(n) \beta_{j_1,\ldots,j_r}(m') w(n/X) - \frac{\mathfrak{S}'(N+1)}{\psi_{N+1}(d)} Y_{j_1,\ldots,j_r} \right|, \\
&+ \ln{X} \sum_{\substack{m' \leq X \\ m' \equiv 1-N \bmod d \\ (d,m')>1}} \beta_{j_1,\ldots,j_r}(m') + X^{1/2+\varepsilon}. 
\end{align*}
For each $d \leq X$, we have
\[ \sum_{N \leq X} \sum_{\substack{m' \leq X \\ m' \equiv 1-N \bmod d \\ (d,m')>1}} \beta_{j_1,\ldots,j_r}(m') \ll \frac{X}{d} \sum_{\substack{m' \leq X \\ (d,m')>1}} \beta_{j_1,\ldots,j_r}(m'), \]
and
\begin{align*}
\sum_{\substack{m' \leq X \\ (d,m')>1}} \beta_{j_1,\ldots,j_r}(m') &\ll \sum_{1 \leq j \leq r} \sum_{\substack{n_i>X^{1/M} \\ n_1 \cdots n_r \leq X \\ (d,n_j)>1}} \Lambda(n_1) \cdots \Lambda(n_r) \\
&\ll (\ln{X})^{r-1} \sum_{n_2 \cdots n_r \leq X^{1-1/M}} \sum_{\substack{X^{1/M} < n_1 \leq X/(n_2 \cdots n_r) \\ (d,n_1)>1}} \Lambda(n_1) \ll \tau(d) X^{1-1/M} (\ln{X})^{O_r(1)}. 
\end{align*}
This and Theorem \ref{T1p} shows that the error term in \eqref{S2r} is small for almost all $N$. 
\end{proof}

For a function $F:[0,1] \to \mathbb{R}_{\geq 0}$ and $s,\theta>0$, we write
\begin{align*}
M_k(F,\theta,s) &= s \omega^-(s\theta) - s \int_{1/s \leq u \leq \theta} \frac{F(u) \omega^+(s(\theta-u))}{u} du \\
& - \frac{2}{\theta} \sum_{r=k+1}^{\infty} \frac{1}{r!} \int \cdots \int_{\substack{1/s \leq u_i \leq 1 \\ u_1 + \cdots + u_r = 1}} \frac{G(u_1, \ldots, u_r)}{u_1 \cdots u_r} du
\end{align*}

By Lemmas \ref{t1S1} and \ref{t1S2}, we have the following proposition. 
\begin{prop}\label{propt1}
Let $k \geq 2$, let $\theta<1/6$ and $s>1/\theta$, and let $F$ be a piecewise smooth non-negative function supported on $[0,\theta-\varepsilon]$ for some $\varepsilon>0$. Assume that $M_k(F,\theta,s)>0$.
Then, for all $A>0$, $X \geq 2$ and even integers $N \in (X/2,X]$ with $O(X(\ln{X})^{-A})$ exceptions, there exist primes $p_1,p_2$ with $p_1+p_2=N$ and $p_1-p_2+1 \in \mathbb{P}_k$. 
\end{prop}

We now prove Theorem \ref{T2}. 

\begin{proof}[Proof of Theorem \ref{T2}]
By Proposition \ref{propt1}, it suffices to find some $\theta<1/6$, $s>2$ and a piecewise smooth function $F$ supported on $[0,\theta)$ such that $M_4(F,\theta,s)>0$.

Let $\theta=0.1635$ and $t= 9.5< s= 33.5$. Define $F$ by
\[ F(u) = \frac{\1_{u \leq 1/t}}{2}. \]
Then
\begin{equation} G(u_1,\ldots,u_r) = \1_{1/t < u_1 \leq \ldots \leq u_r} + \frac{1}{2} \1_{u_1 \leq 1/t < u_2 \leq \ldots \leq u_r} \label{Ghalf} \end{equation}
for $1/s \leq u_1 \leq \cdots \leq u_r \leq 1$ with $u_1+\cdots+u_r=1$. 

By numerical calculation, one has
\[ \omega^-(s\theta) \in[0.56118,\ 0.56121]. \]

\[ \int_{1/s \leq u \leq \theta} \frac{F(u) \omega^+(s(\theta-u))}{u} du \in[0.40941,\ 0.40943]. \]

Define $T_r$ by
\[ T_r = \int \cdots \int_{\substack{1/s \leq u_1 \leq \cdots \leq u_r \leq 1 \\ u_1 + \cdots + u_r = 1}} \frac{G(u_1, \ldots, u_r)}{u_1 \cdots u_r} du. \]
Then we see that
\[ T_5 \in [0.3687,\ 0.3689] \]
\[ T_6 \in [0.04192,\ 0.04193] \]
\[ T_7 \in [0.001818,\ 0.001819] \]
\[ T_8 \in [1.913{\cdot}10^{-5},\ 1.914{\cdot}10^{-5}] \]
\[ T_9 \in [1.327{\cdot}10^{-8},\ 1.329{\cdot}10^{-8}] \]
\[ T_{10} \in [1.017{\cdot}10^{-16},\ 1.023{\cdot}10^{-16}] \]
and $T_{r}=0$ for $r>10$ since $1/s + (r-1)/t>1$. 
Therefore, 
\[ M_4(F,\theta,s) \in [0.0373,\ 0.0385]. \]

\end{proof}

\subsection{Primes with $p_1+p_2=N, 2p_1 p_2+1 \in \mathbb{P}_k$}

We write
\[ \Sigma(f_1,f_2,f_3) = \sum_{m+n=N} f_1(m) f_2(n) f_3(2mn+1) w(n/X) \]

Let $k \geq 2$ be an integer.
Let $F$ be a piecewise smooth non-negative function supported on $[0,\theta/2-\varepsilon]$ for some $\varepsilon>0$, and define $W$ and $W_{>k}$ by
\[ W(n) = 1-\sum_{p|n} F\left(\frac{\ln{p}}{\ln{X^2}}\right), \]
\[ W_{>k}(n) = 
\begin{cases}
0 & (n \in \mathbb{P}_k)  \\
\max(0,W(n)) & (n \notin \mathbb{P}_k) \\
\end{cases} \]
Let $z=(X^{2})^{1/s}, s>2$ and let
\[ S_1 = \Sigma(\Lambda',\Lambda',W \varpi(\cdot,z)), \quad S_2 = \Sigma(\Lambda',\Lambda',W_{>k}\varpi(\cdot,z)). \]

The proof of Theorem \ref{T3} is reduced to showing
\[ S_1 - S_2 > 0 \]
for almost all $6|N$, $k=13$ and some $s,F$.

\begin{lem}\label{Ldim1}
For all but $O(X e^{-(\ln{X})^{1/100}})$ integers $N \in (X/2,X]$, the following holds. 

For any $2<w<y \leq X^{1/4}$, we have
\begin{equation} \sum_{w \leq p < y} \frac{1}{p} \left( \frac{2+N^2}{p} \right) \leq \ln{\left( 1+ \frac{(\ln{X})^{1/20}}{\ln{w}} \right)}. \label{Ldim1eq} \end{equation}
Here, $\left( \frac{\cdot}{p} \right)$ denotes the Legendre symbol modulo $p$.
\end{lem}

\begin{proof}

We first show that 
\begin{equation} \sum_{i: e^{(\ln{X})^{1/50}} \leq 2^i< X^{1/4}} \left| \sum_{2^i \leq p < 2^{i+1}} \frac{1}{p} \left( \frac{2+N^2}{p} \right) \right| \ll e^{-(\ln{X})^{1/100}} \label{pdya} \end{equation}
for all but $O(X e^{-(\ln{X})^{1/100}})$ integers $N\in (X/2,X]$.  

Let $e^{(\ln{X})^{1/50}} \leq t< X^{1/4}$. Then,
\[ \sum_{N \leq X} \left| \sum_{t \leq p < 2t} \frac{1}{p} \left( \frac{2+N^2}{p} \right) \right|^2 = \sum_{t \leq p,p' < 2t} \frac{1}{p p'} \sum_{N \leq X} \left( \frac{2+N^2}{p} \right) \left( \frac{2+N^2}{p'} \right). \]
By the Chinese remainder theorem and a standard bound for the character sums, we see that the sum over $N$ is
\[ \frac{X}{pp'} \left( \sum_{a=1}^{p} \left( \frac{2+a^2}{p} \right) \right) \left( \sum_{a'=1}^{p'} \left( \frac{2+a'^2}{p'} \right) \right) + O(pp') \ll \frac{X}{pp'}, \]
if $p \neq p'$. Thus,
\[ \sum_{N \leq X} \left| \sum_{t \leq p < 2t} \frac{1}{p} \left( \frac{2+N^2}{p} \right) \right|^2 \ll \sum_{t \leq p,p' < 2t} \frac{1}{p p'} \left( \1_{p=p'} X + \frac{X}{pp'} \right) \ll X e^{-(\ln{X})^{1/50}}. \]
This gives \eqref{pdya} for almost all $N$.

Let $N \in (X/2,X]$ and assume \eqref{pdya}.

We divide the sum over $p$ into dyadic ranges $2^i \leq p < 2^{i+1}$, and bound trivially by Mertens theorem for the ranges $p<e^{(\ln{X})^{1/50}}$ and boundary intervals. 
This shows the left-hand side of \eqref{Ldim1eq} is bounded by
\[ \1_{w\leq e^{(\ln{X})^{1/50}}} \ln{(C(\ln{X})^{1/50})} + \ln{\left( 1 + \frac{C}{\ln{w}}\right)} + C e^{-(\ln{X})^{1/50}} \]
for some absolute constant $C>0$. This completes the proof. 
\end{proof}

\begin{lem}\label{St2}
Let $A>0$. 
Assume that $s>8, \theta \in (1/s,1/6)$ and
\[ F(u_1) + \cdots + F(u_r) \geq 1 \]
for all $r>k, u_i \geq 1/s, u_1+\cdots+u_r \leq 1$.

\[ S_1 - S_2 \geq c_w e^{\gamma} X V_N(z) \mathfrak{S}(N) \left(\omega^{-}(s\theta/2) - \int_{1/s}^{\theta/2} \frac{F(u) \omega^{+}(s(\theta/2-u))}{u} du +o(1) \right) \]
for all integers $N \in (X/2,X], 6|N$ with $O(X(\ln{X})^{-A})$ exceptions. 
Here $V_N(z)$ satisfies
\[ V_N(z) \gg \frac{1}{(\ln{z})^2}. \]
\end{lem}

\begin{proof}

By the assumption on $F$, we see that $W_{>k}(n) \neq 0$ only if $p^2|n$ for some $p>z$. Therefore, we have
\[ S_{2} \ll \frac{X}{(\ln{X})^3}. \]

Let 
\[ \Omega_N(d) = \{ l \bmod d : 2l(N-l)+1 \equiv 0 \bmod d \} , \quad g_N(d) = \begin{dcases} \frac{|\Omega_N(d)|}{\psi_N(d)} & 2\nmid d, \\ 0 & 2|d. \end{dcases} \]
By Chinese remainder theorem, we see that $g_N$ is multiplicative with respect to $d$. 

For $p>2$, we have 
\begin{align*}
|\Omega_N(p)| = 1+\left( \frac{2+N^2}{p} \right)
\end{align*}
since the equation is equivalent to $(2x-N)^2 \equiv 2+N^2 \bmod p$. 
In particular, we have $|\Omega_N(d)| \leq \tau(d) \quad (\mu(d)^2=1)$, $g_N(3) = 0$ and $g_N(p) \leq 2/(p-2)$ for $p>3$. 

By Mertens' theorem, we have
\[ \sum_{w\leq p<y} g_N(p) = \ln{\left(\frac{\ln{y}}{\ln{w}}\right)} + \sum_{w\leq p<y} \frac{1}{p} \left( \frac{2+N^2}{p} \right) + O((\ln{w})^{-1}) \]
for all $2<w<y \leq z$. 
This and Lemma \ref{Ldim1} give \eqref{dim1} with $L=(\ln{X})^{1/16}$ for almost all $N$. 

Arguing as in the proof of Lemma \ref{t1S1}, we have

\[ S_1 \geq c_w X e^{\gamma} V_N(z) \mathfrak{S}(N) \left(\omega^{-}(s\theta/2) - \int_{1/s}^{\theta/2} \frac{F(u) \omega^{+}(s(\theta/2-u))}{u} du +o(1) \right) + O\left( \sum_{d \leq X^{\theta}} \left| R_{d} \right| \right), \]
where
\[ V_N(z) = \prod_{p < z} \left( 1- g_N(p) \right) \geq \prod_{5<p < z} \left( 1- \frac{2}{p-2} \right) \gg \frac{1}{(\ln{X})^2} \]
and
\[ R_d = \sum_{\substack{m+n=N \\ d|2mn+1}} \Lambda(m) \Lambda(n) w(n/X) - |\Omega_N(d)| \frac{\mathfrak{S}(N)}{\psi_N(d)} \sum_{n} w(n/X). \]

Since
\[ \sum_{\substack{m+n=N \\ d|2mn+1}} \Lambda(m) \Lambda(n) w(n/X) = \sum_{l \in \Omega_N(d)} \sum_{\substack{m+n=N \\ m \equiv l \bmod d}} \Lambda(m) \Lambda(n) w(n/X) \]
and $(d,l)=(d,N-l)=1$ for all $l \in \Omega_N(d)$, we have
\[ \sum_{d \leq X^{\theta}} \mu(d)^2 \left| R_{d} \right| \ll \sum_{\substack{d \leq X^{\theta} \\ 2 \nmid d}} \tau(d) \sup_{l:(d,l(N-l))=1} \left| \sum_{\substack{m+n=N \\ m \equiv l \bmod d}} \Lambda(m) \Lambda(n) w(n/X) - \frac{\mathfrak{S}(N)}{\psi_N(d)} \sum_{n} w(n/X) \right| \ll \frac{X}{(\ln{X})^3} \]
for almost all $N$ by Theorem \ref{T1}. This completes the proof.

\end{proof}

\begin{proof}[Proof of Theorem \ref{T3}]

Let $k=13$, and define
\[ F(u) = \frac{\1_{u \leq \lambda} }{(k+1)\lambda - 1} \left( \lambda - u \right), \]
where
\[ \gamma = \frac{1}{12}, \quad \lambda = \frac{\gamma}{1+3^{-k}}, \quad s = \frac{4}{\gamma}>8. \]
The computation in \cite[Section 6]{heath2002lectures} shows that
\[ \omega^{-}(s/12) - \int_{1/s}^{1/12} \frac{F(u) \omega^{+}(s(1/12-u))}{u} du > 0, \]
since
\[ 0.0833 \ldots = \frac{1}{12} > \frac{1}{k+1 - \frac{\ln{4/(1+3^{-k})}}{\ln{3}}} = 0.0785\ldots. \]
This and Lemma \ref{St2} complete the proof. 
\end{proof}

\section{Appendix}

In this appendix, we record numerical computations for the weighted sieve criterion, independently of the analytic distribution estimates proved in the main text. The computations are inspired by the numerical work of Matom{\"a}ki and Z{\'u}{\~n}iga-Alterman \cite{matomaki2025weighted}. 
The case \(k=4\) supplies the numerical input used in the proof of Theorem \ref{T2}. 

The initial versions of the arguments and numerical computations in this appendix were generated with OpenAI Codex. The material used in the proof of Theorem \ref{T2}, as well as the computations reported in Table 1 below, was subsequently revised and carefully checked by the author. These certified numerical computations were carried out in C using Arb and FLINT. By contrast, Table 2 is included only as a numerical Richert-type comparison; it has not been checked by the author to the same standard, and it has not been certified using Arb/FLINT interval arithmetic. The source code is available in the arXiv version of this manuscript.

\subsection{Setting}

Recall that for a function $F:[0,1] \to \mathbb{R}_{\geq 0}$ and $s,\theta>0$, we write
\begin{align*}
M_k(F,\theta,s) &= s \omega^-(s\theta) - s \int_{1/s \leq u \leq \theta} \frac{F(u) \omega^+(s(\theta-u))}{u} du \\
& - \frac{2}{\theta} \sum_{r=k+1}^{\infty} \frac{1}{r!} \int \cdots \int_{\substack{1/s \leq u_i \leq 1 \\ u_1 + \cdots + u_r = 1}} \frac{G(u_1, \ldots, u_r)}{u_1 \cdots u_r} du
\end{align*}
where
\[ G(u_1,\ldots,u_r) = \max(0,1-F(u_1)-\cdots-F(u_r)). \]

We want to find small $\theta \in (0,1)$ such that $M_k(F,\theta,s)>0$ for some non-negative piecewise smooth function $F$ supported on $[0,\theta]$ and $s>1/\theta$.

We write
\[ S_G = 
\sum_{r=k+1}^{\infty} T_r, \quad T_r = 
\int_{\substack{1/s\le u_1\le \cdots\le u_r\le 1\\
u_1+\cdots+u_r=1}}
\frac{G(u_1,\ldots,u_r)}{u_1\cdots u_r}\,
du.
\]
and
\[
A^\pm(u):=u\omega^\pm(u).
\]
The linear sieve functions are normalized by
\[
A^+(u)=2 \quad (0<u\le 3),\qquad
A^-(u)=0 \quad (0<u\le 2),
\]
and by the delay equations
\[
(A^+)'(u)=\frac{A^-(u-1)}{u-1}\quad (u>3),\qquad
(A^-)'(u)=\frac{A^+(u-1)}{u-1}\quad (u>2).
\]

Following \cite{matomaki2025weighted}, define
\[
c_1(s)=\1_{s\ge1},\qquad
c_r(s)=\int_{r \leq u \leq s} \frac{c_{r-1}(u-1)}{u-1}\,du
\quad (r\ge2).
\]
We have
\[
c_r(s)=
\int_{\substack{1/s\le u_1\le\cdots\le u_r\\
u_1+\cdots+u_r=1}}
\frac{du}{u_1\cdots u_r}.
\]
for $r \geq 2$. This follows by direct calculation if $r=2$. Assuming the case $r$, the case $r+1$ follows from
\begin{align*}
&\int_{1/s\le u_1\le\cdots\le u_r \leq 1 - u_1 - \cdots - u_{r}} \frac{du_1 \cdots du_r}{u_1\cdots u_{r} (1 - u_1 - \cdots - u_r)} \\
&= \int_{1/s \leq u_1 \leq 1/(r+1)} \frac{c_r((1-u_1)/u_1)}{u_1(1-u_1)} du_1 = c_{r+1}(s). 
\end{align*}
Here the first equality follows by writing
\(u_{i+1}=(1-u_1)v_i\) for \(1\le i\le r\), and the last equality follows from the change of variables \(x=1/u_1\).

Three choices of \(F\) have been used in this context; see \cite{matomaki2025weighted} for the history and applications. 

\par
\text{(Trivial weight)} $F=0$, \par
\text{(Kuhn's weight)} $F(u)=\frac12\1_{u\le1/t}$ for some $t<s$. \par
\text{(Richert's weight)} $F(u)=\lambda\1_{u\le1/t}(1-tu)$ for some $t<s$ and $\lambda>0$. \par

For the trivial weight, one has \(G=1\), and hence
\[ T_r= c_{r}(s). \]

For Kuhn's weight, one has
\[ T_r=T_{r,0}+T_{r,1}, \qquad T_{r,0} =c_r(t), \quad T_{r,1} = \frac{t}{2}\int_{t-1}^{t(1-1/s)} \frac{c_{r-1}(u)}{u(t-u)}du \]
by \eqref{Ghalf}. The first integral in \(M_k\) is evaluated using
\[
\frac12\int_{1/s}^{1/t}\frac{\omega^+(s(\theta-u))}{u}\,du
=
\frac12\int_{\beta(1-1/\alpha)}^{\beta-1}
\frac{A^+(v)}{v(\beta-v)}\,dv,
\qquad
\alpha=t\theta,\quad \beta=s\theta.
\]

These two weights only require values of \(A^\pm\) and \(c_r\), which can be computed rigorously by the recursion formulas above. 

For Richert's weight, \(S_G\) was computed numerically by an FFT convolution method. This computation was not enclosed using Arb interval arithmetic.

\subsection{Table}

The next table was computed in C using FLINT and Arb interval arithmetic. The source code was carefully checked by the author.
The intervals have been rounded outwards for readability.

\begin{table}[ht] 
\centering
\small
\setlength{\tabcolsep}{5pt}
\begin{tabular}{@{}llrrrrl@{}}
\toprule
weight & \(k\) & \(\theta\) & \(s\) & \(t\) & \(S_G\) & certified interval for \(M\)\\
\midrule
\(0\)
 & 2 & 0.50512 & 5.12  & --    & 0.4588 & \([0.0097,\ 0.0105]\)\\
 & 3 & 0.29515 & 8.76  & --    & 0.4594 & \([0.0095,\ 0.0108]\)\\
 & 4 & 0.18352 & 14.17 & --    & 0.4693 & \([0.0097,\ 0.0119]\)\\
 & 5 & 0.11801 & 22.18 & --    & 0.4802 & \([0.0102,\ 0.0136]\)\\
 & 6 & 0.07742 & 34.02 & --    & 0.4905 & \([0.0095,\ 0.0147]\)\\
 & 7 & 0.05145 & 51.50 & --    & 0.5000 & \([0.0170,\ 0.0249]\)\\
\addlinespace
\(\frac12\1_{u\le1/t}\)
 & 2 & 0.47337 & 10.87 & 3.35  & 0.4205 & \([0.0094,\ 0.0109]\)\\
 & 3 & 0.26631 & 20.25 & 5.83  & 0.4090 & \([0.0090,\ 0.0117]\)\\
 & 4 & 0.16328 & 33.53 & 9.56  & 0.4215 & \([0.0107,\ 0.0117]\)\\
 & 5 & 0.10438 & 52.76 & 15.08 & 0.435  & \([0.0097,\ 0.0166]\)\\
 & 6 & 0.06832 & 80.79 & 23.25 & 0.448  & \([0.0057,\ 0.0163]\)\\
 & 7 & 0.04538 & 121.73& 35.30 & 0.4597 & \([0.0135,\ 0.0295]\)\\
\bottomrule
\end{tabular}
\caption{Certified positive values of \(M\).}
\end{table}

For Richert's weight, the following values are included only as reference numerical values. They are \textbf{not} used in the proof of Theorem \ref{T2}; they are \textbf{not} certified using Arb interval arithmetic; and they have \textbf{not} been checked by the author to the same standard as Table~1. 
The source code for this computation is also included in the arXiv version of this manuscript.

\begin{table}[ht]
\centering
\small
\setlength{\tabcolsep}{6pt}
\begin{tabular}{@{}crrrrcc@{}}
\toprule
\(k\) & \(\theta\) & \(s\) & \(t\) & \(\lambda\) & \(S_G\) & exact \(r\)-range\\
\midrule
2 & 0.48420 & 9.48   & 2.95  & 1.02 & 0.4976 & \(3\)\\
3 & 0.27318 & 18.28  & 4.46  & 0.86 & 0.4340 & \(4\text{--}5\)\\
4 & 0.16695 & 30.79  & 7.12  & 0.81 & 0.4535 & \(5\text{--}8\)\\
5 & 0.10638 & 48.70  & 11.04 & 0.78 & 0.4706 & \(6\text{--}12\)\\
6 & 0.06946 & 75.71  & 16.82 & 0.77 & 0.4884 & \(7\text{--}18\)\\
7 & 0.04605 & 113.81 & 25.51 & 0.76 & 0.5013 & \(8\text{--}26\)\\
\bottomrule
\end{tabular}
\caption{Richert-type rows; reference numerical values only, not certified.}
\end{table}

For \(k=4\) and \(\theta=1/6\), we could not find a choice of parameters in this Richert family giving $M_k>-0.025$.

\FloatBarrier

Among the parameter choices displayed in these tables, Kuhn's weight gives the smallest value of $\theta$.

\subsection{Asymptotic}

The following proposition gives an asymptotic upper bound for $\theta$ obtainable from the \(F=0\) construction.
For comparison, it is known that, without switching, one needs a condition of order at least \(1/k\) to find \(\mathbb{P}_k\). 
\begin{prop}
Let $\varepsilon>0$, and assume that $k$ is sufficiently large in terms of $\varepsilon$. 
Then there exists $s_k>0$ such that $M_k(0,\theta_k,s_k)>0$ where
\[
\theta_k=\exp\!\left(-(1-\varepsilon)\frac{k}{e}\right).
\]
\end{prop}

\begin{proof}

We first note
\begin{equation} \label{cjb} c_j(s) \leq \frac{(\ln{s})^{j-1}}{(j-1)!}. \end{equation}
If $j=1$, this is trivial. Assuming the case $j-1$, the case $j$ follows from
\[ c_j(s) \leq \int_{j\leq u \leq s} \frac{(\ln{(u-1)})^{j-2}}{(j-2)! (u-1)} du = \frac{(\ln{(u-1)})^{j-1}}{(j-1)!} \Big|_{u=j}^{s} \leq \frac{(\ln{s})^{j-1}}{(j-1)!}. \]

Hence we have
\[ C_{k+1}(s) \leq \sum_{r\geq k} \frac{(\ln{s})^r}{r!}. \]
Since $\ln{r!} \geq r(\ln{r}-1)$, we have $(\ln{s})^r/r! \leq \left( e \ln{s}/r \right)^r$. 

We take
\[ s_k = 3/\theta_k. \]
We may assume $\ln{s_k} \leq (1-\varepsilon/2)k/e$ if $k$ is large. Hence
\[ C_{k+1}(s_k) \leq \sum_{r \geq k} (1-\varepsilon/2)^r \leq \frac{(1-\varepsilon/2)^k}{\varepsilon/2} \to 0 \quad (k \to \infty), \]
and
\[ M_k(0,\theta_k,s_k) \geq s_k \left( \omega^-(3) - \frac{2}{3} C_{k+1}(s_k) \right) >0 \]
for large $k$. 
\end{proof}

\bibliographystyle{plain}
\nocite{*}
\bibliography{ref.bib}

\end{document}